\documentclass[leqno,11pt,a4paper]{amsart} 
\makeatletter
\renewcommand\section{\@startsection{section}{1}{\z@}%
           {25\p@ \@plus 6\p@ \@minus 3\p@}%
           {10\p@ \@plus 6\p@ \@minus 3\p@}%
           {\fontsize{13pt}{0cm}\selectfont\bfseries\boldmath}}
\renewcommand\subsection{\@startsection{subsection}{2}{\z@}%
           {13\p@ \@plus 6\p@ \@minus 3\p@}%
           {6\p@ \@plus 6\p@ \@minus 3\p@}%
           {\fontsize{12pt}{0cm}\itshape}}
\renewcommand\subsubsection{\@startsection{subsubsection}{3}{\z@}%
           {12\p@ \@plus 6\p@ \@minus 3\p@}%
           {\p@}%
           {\normalfont\normalsize\itshape}}
\renewcommand{\paragraph}[1]{%
  \par
  \addvspace{\medskipamount}
  \noindent
  \textbf{#1\@addpunct{.}}\enspace\ignorespaces
}
\makeatother
\let\oldtocsection=\tocsection
\let\oldtocsubsection=\tocsubsection
\let\oldtocsubsubsection=\tocsubsubsection
\renewcommand{\tocsection}[2]{
\hspace{0em}\bfseries\oldtocsection{#1}{#2}}
\renewcommand{\tocsubsection}[2]{\hspace{1em}\small\oldtocsubsection{#1}{#2}}
\renewcommand{\tocsubsubsection}[2]{\hspace{2em}\small\oldtocsubsubsection{#1}{#2}}
\usepackage[colorlinks=true,urlcolor=blue,citecolor=red,linkcolor=blue,linktocpage,pdfpagelabels,bookmarksnumbered,bookmarksopen]{hyperref}
\usepackage{amssymb, amsmath, amsthm, physics}
\usepackage{amsfonts} 
\usepackage{color}
\usepackage{mathrsfs}
\usepackage{graphicx}
\usepackage{comment}
\usepackage{latexsym, mathtools}
\usepackage{xparse}
\usepackage{layout}
\usepackage{fancyhdr}
\usepackage[capitalize,nameinlink]{cleveref}
\crefname{subsection}{Subsection}{Subsection}
\usepackage{enumerate}
\DeclareSymbolFont{EulerExtension}{U}{euex}{m}{n}
\DeclareMathSymbol{\euintop}{\mathop} {EulerExtension}{"52}
\DeclareMathSymbol{\euointop}{\mathop} {EulerExtension}{"48}

\DeclareSymbolFont{cmletters}{OML}{cmm}{m}{it}
\DeclareSymbolFont{cmsymbols}{OMS}{cmsy}{m}{n}
\DeclareSymbolFont{cmlargesymbols}{OMX}{cmex}{m}{n}
\DeclareMathSymbol{\myjmath}{\mathord}{cmletters}{"7C}
\DeclareMathSymbol{\myamalg}{\mathbin}{cmsymbols}{"71}
\DeclareMathSymbol{\mycoprod}{\mathop}{cmlargesymbols}{"60}
\let\jmath\myjmath

\newtheorem{theorem}{Theorem}[section]
\newtheorem{proposition}[theorem]{Proposition}
\newtheorem{lemma}[theorem]{Lemma}
\newtheorem{corollary}[theorem]{Corollary}

\theoremstyle{definition}
\newtheorem{remark}[theorem]{Remark}
\newtheorem{definition}[theorem]{Definition}

\newtheorem{step}{Step}

\expandafter\let\expandafter\oldproof\csname\string\proof\endcsname
\let\oldendproof\endproof
\renewenvironment{proof}[1][\proofname]{%
  \oldproof[\bfseries\itshape #1] 
}{\oldendproof}
\numberwithin{equation}{section}
\renewcommand{\theequation}{\thesection.\arabic{equation}}

\allowdisplaybreaks[4]

\DeclareMathOperator*{\supp}{supp}

\newcommand{\toinfty}{\to\infty}
   
\newcommand{\eps}{\varepsilon}          
\renewcommand{\l}{\left}
\renewcommand{\r}{\right}
\newcommand{\la}{\left\langle}
\newcommand{\ra}{\right\rangle}
\newcommand{\R}{{\bf R}}
\newcommand{\N}{{\bf N}}
\newcommand{\NN}{{\bf Z}_{\ge 0}}
\newcommand{\Z}{{\bf Z}}
\newcommand{\RN}{\R^N}

\newcommand{\F}{\mathcal{F}}
\newcommand{\FF}{\mathcal{F}^{-1}}

\renewcommand{\d}{\partial}

\newcommand{\intrn}{\int_{\RN}}

\newcommand{\s}{\,\,\,}
\newcommand{\eqntag}{\addtocounter{equation}{1}\tag{\theequation}}

\newcommand{\leb}[2]{L^{#1}(#2)}

\newcommand{\Wdrn}{\dot{W}^{1,p}(\RN)}
\newcommand{\Wmp}{W^{m,p}(\RN)}
\newcommand{\Wdmp}{\dot{W}^{m,p}(\RN)}

\newcommand{\Hdsp}{\dot{H}^{s,p}(\RN)}
   
\newcommand{\p}{p^*}
\newcommand{\ppm}{p^*_m}        
\newcommand{\pps}{p^*_s}        
\newcommand{\dsl}[3]{#1_{#2,#3}}        
\newcommand{\bbl}[2]{#1^{#2}}                   

\newcommand{\Bnrl}{ \mathcal{B}_{n,R,L} }
\newcommand{\Jlp}{\mathcal{J}_{l,L}^+}
\newcommand{\Jln}{\mathcal{J}_{l,L}^-}
\newcommand{\Jlz}{\mathcal{J}_{l,L}^0}
\begin{document}


\title[Profile Decomposition in homogeneous Sobolev spaces]{Profile decomposition in Sobolev spaces and decomposition of integral functionals~II: \\ homogeneous case} 

\author[M.~Okumura]{Mizuho Okumura} 


\renewcommand{\thefootnote}{\fnsymbol{footnote}}
\footnote[0]{2020\textit{ Mathematics Subject Classification}.
46B50, 46B06, 46E35.}

\address{
Graduate School of Science \endgraf
Tohoku University \endgraf
Sendai 980-8578 \endgraf
Japan
}
\email{okumura.mizuho.p3@dc.tohoku.ac.jp}




\maketitle

\begin{abstract}
The present paper is devoted to a theory of profile decomposition for bounded sequences in \emph{homogeneous} Sobolev spaces, and it enables us to analyze the lack of compactness of bounded sequences.  For every bounded sequence in homogeneous Sobolev spaces, the sequence is asymptotically decomposed into the sum of profiles with dilations and translations and a double suffix residual term. One gets an energy decomposition in the homogeneous Sobolev norm. The residual term becomes arbitrarily small in the critical Lebesgue or Sobolev spaces of lower order, and then, the results of decomposition of integral functionals are obtained, which are important strict decompositions in the critical Lebesgue or Sobolev spaces where the residual term is vanishing.
\end{abstract}

{\small %
\tableofcontents
}

\section{Introduction}
\subsection{Prologue}
The lack of compactness of bounded (mainly function) sequences in infinite dimensional normed vector spaces has been a significant and critical difficulty in mathematical  analysis.
To overcome the difficulty, a lot of researchers have been establishing a variety of ways to discuss the lack of compactness and the recovery of compactness. 
Classical and well-known  results are, for instance, the Brezis-Lieb lemma~\cite{B-L}, the concentration-compactness principles by Lions~\cite{Lions 84 1, Lions 84 2, Lions 85 1, Lions 85 2}, the global compactness results by Struwe~\cite{Struwe}, Brezis-Coron~\cite{BrezisCoron85} and Bahri-Coron~\cite{BahriCoron88}, and so on. 

Profile decomposition is a way of asymptotic analysis of general bounded sequences that may have the defect of compactness, and originated in attempts to give an asymptotic decomposition of bounded sequences in some function spaces such as Lebesgue spaces, Sobolev spaces and so on. 
It states that every bounded sequence in a certain Banach space is asymptotically decomposed into a sum of \emph{moving} profiles and a residual term, where that \emph{movements} of profiles are considered to be descriptions of the defect of compactness.  

Roughly speaking, there are two types of principles of profile decomposition in terms of the residual term: one with a single-suffix residual term as 
is treated in, e.g., Solimini~\cite{Solimini}, Tintarev-Fieseler~\cite{T-K} and Tintarev~\cite{T-01};
%
another with a double-suffix residual term as 
is treated in, e.g., G\'erard~\cite{Gerard}, Jaffard~\cite{Jaffard} and Bahouri-Cohen-Koch~\cite{B-C-K}.
%
The differences and relations between the above two types of principles are described in, e.g.,~\cite{DevillanovaSolimini,Okumura2}.
Most profile decomposition theorems offered in the above literature focus on isometric group actions on function spaces, named ``dislocations", which fairly prevents bounded sequences from converging strongly. By making use of such group actions appropriately, generic principles of profile decomposition have been successfully established.

Here we briefly review the results of Solimini and Jaffard. 
Solimini~\cite{Solimini} developed the profile decomposition with a single-suffix residual term in the following sense. For any bounded sequence $(u_n)$ in $\dot{W}^{1,p}(\RN)$  $(1<p<N)$, there exist $\bbl{w}{l}\in \dot{W}^{1,p}(\RN)$ $(l\in\NN)$ and $(\dsl{y}{l}{n},\dsl{j}{l}{n})\in\RN\times\R$ $(l,n\in\NN)$ such that, up to a subsequence, 
\begin{align*}\eqntag\label{Solimini1}
u_n = \sum_{l=0}^\infty 2^{\dsl{j}{l}{n}N/\p} \bbl{w}{l} ( 2^{\dsl{j}{l}{n}} (\cdot -\dsl{y}{l}{n}) ) + r_n, \quad n\in\NN, 
\end{align*}
with 
\begin{align*}
&
2^{-\dsl{j}{l}{n}N/\p} u_n ( 2^{-\dsl{j}{l}{n}}\cdot +\dsl{y}{l}{n} )\to \bbl{w}{l}
\s\mbox{weakly in}\ \Wdrn, 
\\
&
|\dsl{j}{l}{n}-\dsl{j}{k}{n}| + 2^{\dsl{j}{k}{n}}|\dsl{y}{l}{n}-\dsl{y}{k}{n}|\toinfty
\s\mbox{if}\ l\neq k \s(n\toinfty),
\\
&
\lim_{n\toinfty} \norm{r_n}_{\leb{\p,q}{\RN}}=0, \quad  q>\p,
\\
&
\limsup_{n\toinfty} \norm{u_n}^p_{\Wdrn} \ge 
\sum_{l=0}^\infty \norm{\bbl{\phi}{l}}^p_{\Wdrn}, 
\end{align*}
where $\p\coloneqq pN/(N-p)$ denotes the Sobolev critical exponent and $\leb{\p,q}{\RN}$ denotes the Lorentz space.

On the other hand, Jaffard~\cite{Jaffard} developed the profile decomposition with a double-suffix residual term in the following sense. For any bounded sequence $(u_n)$ in $\dot{H}^{s,p}$, there exist $\bbl{w}{l}\in\dot{H}^{s,p}(\RN)$ $(l\in\NN)$ and $(\dsl{y}{l}{n},\dsl{j}{l}{n})\in\RN\times \R$ $(l,n\in\NN)$ such that, up to a subsequence,  
\begin{align*}\eqntag\label{Jaffard1}
u_n = \sum_{l=0}^L 2^{\dsl{j}{l}{n}N/\p} \bbl{w}{l} ( 2^{\dsl{j}{l}{n}} (\cdot -\dsl{y}{l}{n}) ) + r^L_n, \quad L,n\in\NN, 
\end{align*}
with 
\begin{align*}
&
|\dsl{j}{l}{n}-\dsl{j}{k}{n}| + 2^{\dsl{j}{k}{n}}|\dsl{y}{l}{n}-\dsl{y}{k}{n}|\toinfty
\s\mbox{if}\ l\neq k \s(n\toinfty),
\\
&
\lim_{L\toinfty}\varlimsup_{n\toinfty} \norm{r^L_n}_{\leb{\pps}{\RN}}=0, 
\\
&
\limsup_{n\toinfty}\norm{u_n}^p_{\Hdsp}
\ge \sum_{l=0}^\infty \norm{\bbl{\phi}{l}}^p_{\Hdsp}. 
\end{align*}
For more details of profile decomposition and applications of it, we refer the reader to~\cite{B-C-K, Gerard, Jaffard, Okumura2, Solimini, tao blog-book, T-01, T-K} and references therein. 
Also the relations between profile decomposition and classical ways of analysis of the lack of compactness are shown in~\cite{Okumura2}. 

The theory of Solimini and its extensions to Banach and Hilbert spaces by Tintarev et al.\ are successful in characterizing the profiles $\bbl{w}{l}$ with $u_n$ by the use of the weak convergence.
However, the well-definedness or the convergence of the infinite sum on~\eqref{Solimini1} matter delicately and need some complicated arguments for verification, named ``\emph{routing procedure}'' in~\cite{DevillanovaSolimini}. 
On the other hand, in the theory of G\'erard, Jaffard and Bahouri et al.\ appears the finite sum of scaled profiles (also called \emph{dislocated profiles}) and the well-definedness and convergence do not matter. 
Also, Jaffard~\cite[p.386]{Jaffard} remarks that the finite sum on~\eqref{Jaffard1} cannot be replaced with the infinite sum as in~\eqref{Solimini1} because it does not usually converge.

\paragraph{Aim of the paper}
Under these circumstances, in~\cite{Okumura2} the author developed a profile decomposition theorem in \emph{inhomogeneous} Sobolev spaces, which was 
a \emph{hybrid type} profile decomposition of~\cite{Solimini, T-01, T-K} and~\cite{B-C-K,Gerard, Jaffard}. In his theory, each profile of a bounded sequence is well characterized by the weak convergence and an isometric group action on inhomogeneous Sobolev spaces, that is, \emph{translations} as in~\cite{Solimini, T-01, T-K}, but only a finite sum appears so that the well-definedness and convergence are always valid as in~\cite{B-C-K,Gerard, Jaffard}.
He also investigated significant results of decomposition of integral functoinals (also regarded as the iterated Brezis-Lieb lemma and investigated by, e.g.,~\cite{T-K,T-01}), revealing that the profile decomposition leads to an asymptotic strict decomposition of integral functionals in suitable Lebesgue or Sobolev spaces where the residual term of profile decomposition is vanishing. 

In the present paper, we shall develop a ``\emph{homogeneous}" version of a theory of profile decomposition and decomposition of integral functionals in contrast to~\cite{Okumura2}, proving them in the same spirit.
As for the energy decomposition in the Sobolev norm (see~\eqref{eq;energy estim Wdmp} below), it will be a sharper version of the ordinary energy decompositions provided by precursors which do not include the ``residual term''. 
Results provided below will be well applied to studies of PDEs and Calculus of Variations. 

\subsection{Notation and settings}
As is mentioned above, the way of profile decomposition needs a  suitable setup of group actions responsible for the lack of compactness. Those group actions are called ``\emph{dislocations}" and they are built up from functional analytic viewpoints inspired by~\cite{T-01,T-K}.

\smallskip
Throughout the paper, we often use the following  
notation. 
%
\begin{enumerate}
\setlength{\itemsep}{0mm}
\setlength{\parskip}{0mm}

    \item We write $(n_j) \preceq (n)$ when the left-hand side is a subsequence of the right-hand side. 
    \item The set of integers greater than or equal to $l \in \Z$ is denoted by $\Z_{\ge l}$.
    \item For $\Lambda \in \N \cup \{0,+\infty\}$, the set of integers at least zero and at most  $\Lambda$ is denoted by $\NN^{<\Lambda+1}$, where we assume that $\Lambda +1 = +\infty$ when $\Lambda = +\infty$, i.e., 
    $\NN^{<\Lambda+1} = \qty{ 0, \ldots, \Lambda }$ if $\Lambda <+\infty$ and 
    $\NN^{<\Lambda+1} = \NN$ if $\Lambda = +\infty$. 
    Moreover, for real numbers $(a_L)_L$, we often write  $\lim_{L\to\Lambda} a_L = a_\Lambda$ whenever $\Lambda<\infty$. 
    \item We denote by $C$ a non-negative constant, which does not
depend on the elements of the corresponding space or set and may vary from line to line.
    \item For an exponent $p \in [1,\infty]$,  we denote by $p'$ the  H\"{o}lder conjugate exponent of $p$: $p'=p/(p-1)$ if $p\in]1,\infty[;$ $p'=\infty$ if $p=1;$ $p'=1$ if $p=\infty$. 
    \item For a normed space $X$,  we denote its norm by $\|\cdot\|$ or $\|\cdot\|_{X}$, 
    denote its dual space by $X^*$ and 
    denote the duality pairing between $X,X^*$ by $\la \cdot,\cdot\ra$ or $\la\cdot,\cdot\ra_X$.
    \item For a normed space $X$, we denote by $B_X(r)$ 
    the closed ball of radius $r>0$ centered at the origin, that is, 
    $B_X(r) = \qty{ u \in X; \ \|u\|_X \le r }$. 
    Meanwhile, the closed ball in $\RN$ of radius $r>0$ centered at $p\in \RN$ is denoted by $B(p,r)$, that is, 
    $B(p,r)\coloneqq \qty{x\in \RN; \ |x-p|\le r}$. 
    \item For a normed space $X$,  we denote by $\mathcal{B}(X)$ 
    the normed space of all bounded linear operators on $X$ equipped with the operator norm. 
    \item For a Banach space $X$ and for a
    bounded operator $T\in \mathcal{B}(X)$,  the adjoint operator
    of $T$ is denoted by $T^*$.
\end{enumerate}

\smallskip
We shall follow and employ the setup for the framework of profile decomposition as in~\cite{Okumura2}, and we here recall them briefly. 
Let $(X,\|\cdot\|_X)$ be a Banach space and let $G\subset\mathcal{B}(X)$ be a group (under operator composition) of bijective isometries. 

\begin{definition}\label{def;DslSp}
\begin{enumerate}
\setlength{\itemsep}{0mm}
\setlength{\parskip}{0mm}

\item (\emph{Operator convergence}) For a sequence $(g_n)$ in $\mathcal{B}(X)$, the operator-strong convergence of $(g_n)$ is defined as the pointwise strong convergence in $X$, while the operator-weak convergence of $(g_n)$ is defined as the pointwise weak convergence in $X$. 

\item (\emph{$G$-weak convergence}) A sequence $(u_n)$ in $X$ is said to be $G$-weakly convergent to $u\in X$ provided that 
\begin{equation*}
    \lim_{n\to\infty}\sup_{g\in G} \qty|\la \phi,g^{-1}(u_n-u)\ra |=0
    \quad\mbox{for all}\ \phi\in X^*.
\end{equation*}

\item (\emph{Dislocation group}) A group $G$ of linear bijective isometries is called a \emph{dislocation group} if the following two conditions are satisfied. 
\begin{enumerate}
    \item For any sequence $(g_n)$ in $G$ with $g_n\not\rightharpoonup 0$, there exists a subsequence $(n_j)\preceq (n)$ such that $(g_{n_j})$ is convergent operator-strongly.
    
    \item For any sequence $(g_n)$ in $G$ with $g_n\not\rightharpoonup 0$ and for any sequence $(u_n)$ in $X$ with $u_n\to 0$ weakly in $X$,  there exists a subsequence $(n_j)\preceq (n)$ such that $g_{n_j}u_{n_j}\to 0$ weakly in $X$.
\end{enumerate}

\item (\emph{Dislocation space}) The pair $(X,G)$ is called a \emph{dislocation space} when $G$ is a dislocation group. 

\item (\emph{$G$-complete continuity}) For a normed space $Y$, a linear operator $T: X\to Y$ is said to be $G$-completely continuous if every $G$-weakly convergent sequence in $X$ is mapped to a strongly convergent sequence in $Y$.  
\end{enumerate}
\end{definition}

Regarding the above settings, we shall make some remarks as in~\cite{Okumura2}:
\begin{remark}\label{rem;DslSp}
\rm 
\begin{enumerate}
\setlength{\itemsep}{0mm}
\setlength{\parskip}{0mm}

\item If $G=\{ {\rm Id}_X \},$ then the $G$-weak convergence is nothing but the ordinary weak convergence, i.e., the $G$-weak convergence is an extension of the weak convergence.  

\item Suppose, in addition, that $G$ satisfies that for every sequence $(g_n)$ in $G$ with $g_n\not\rightharpoonup 0$, the adjoint $(g_n^*)$ has an operator-strongly convergent subsequence. Then the condition~(b) in~(iii) of \cref{def;DslSp} above is always satisfied. 

\item If $(X,\|\cdot\|_1)$ is equipped with another equivalent and complete norm $\|\cdot\|_2$, if the action of $G$ on $X$ is isometric in both $\|\cdot\|_1$ and $\|\cdot\|_2$, and if $((X,\|\cdot\|_1),G)$ is a dislocation space, then the pair $((X,\|\cdot\|_2),G)$ is also a dislocation space. In other words, the definition of dislocation spaces is irrelevant to the choice of equivalent norms. 

\item The $G$-complete continuity is also called $G$-cocompactness in, e.g., \cite{T-01} and some Tintarev's papers.
\end{enumerate}
\end{remark}

For more details of properties or examples of dislocation spaces or $G$-complete continuity, we refer the reader to~\cite{Okumura2, T-01, T-K} and references therein. 

\subsection{Strategy of the paper}
As is developed in~\cite{Okumura2}, we here briefly review a \emph{recipe} of a theory of profile decomposition and decomposition of integral functionals. 
To construct the theory, one should prepare a pair $(X,G,Y)$ where $(X,G)$ is a dislocation Sobolev space which is embedded into some Lebesgue or Sobolev space $Y$ $G$-completely continuously. 
The profile decomposition in $(X,G)$ is obtained in three steps: (i) finding profiles by the use of~\cite[Theorem~2.1]{Okumura2} (which will be exhibited in \cref{section:appendixA} below); (ii) a decomposition in the Sobolev norm and the exactness condition of profile decomposition; (iii) vanishing of the residual term in $Y$. 
The spirit of decomposition of integral functionals is: they are developed in Lebesgue or Sobolev spaces $Y$ into which $(X,G)$ is embedded $G$-completely continuously. 
Here the essential difficulty and difference in homogeneous case  compared to the inhomogeneous case~\cite{Okumura2} are calculations in the energy decomposition in the homogeneous Sobolev norm. 
To this end, we shall employ an appropriate separation of the domain (see \cref{lemma;energy estim Wdmp,lemma;cut off induction,lemma;calc in proof of claim E}).

\subsection{Organization of the paper}
In \cref{section;ProDecoHomSobolev}, we shall establish a theorem of profile decomposition in homogeneous Sobolev spaces, which is our main interest of the paper. 
In \cref{BLcrit}, the results of decomposition of integral functionals subordinated to the profile decomposition in homogeneous Sobolev spaces will be discussed. 
Results and their proofs will be given in the same sections.

\section[Profile decomposition]{Profile decomposition in  homogeneous Sobolev spaces}\label{section;ProDecoHomSobolev}

\subsection{Settings}
We begin with the definition of \emph{homogeneous} Sobolev spaces.
Let $\Omega$ be an open set in $\RN$.
For $1\le p < N/m, \ m, N\in\N$,  the \emph{homogeneous} Sobolev space $\dot{W}^{m,p}(\Omega)$ is 
defined by 
\begin{equation}\notag
\dot{W}^{m,p}(\Omega) =  \overline{C_c^\infty(\Omega)}^{\|\cdot\|_{\dot{W}^{m,p}(\Omega)}},
\end{equation}
where 
\begin{equation}\notag
\|u\|_{\dot{W}^{m,p}(\Omega)}=\l( \sum_{|\alpha|=m} \|\d^{\alpha}u\|_{L^p(\Omega)}^p \r)^{1/p}.
\end{equation}%
Here for a multi-index $\alpha = (\alpha_1, \ldots,\alpha_N) \in (\NN)^N$,  
$$
|\alpha|=\sum_{n=1}^N \alpha_n \mbox{\quad and\quad}
\d^\alpha = \frac{\d^{|\alpha|}}{\d x_1^{\alpha_1}\cdots \d x_N^{\alpha_N}},
$$
and derivatives are in the sense of distributions.
We denote by $\ppm$ the Sobolev critical exponent:
$\ppm=pN/(N-mp)$ if $N>mp;$ $ \ppm =\infty$ if $N\le mp$.

    The reader ought to be careful in reading~\cite{T-01} 
    because Tintarev uses other notation for Sobolev spaces. 
    He denotes the Sobolev spaces defined as above by $\dot{H}^{m,p}(\Omega)$ instead. 
    The assumption $p<N/m$ in considering homogeneous Sobolev spaces is attributed to the fact that otherwise homogeneous Sobolev spaces  are no longer spaces of measurable functions, i.e., there are no longer injections from $\dot{W}^{m,p}(\RN)$ 
    into $L^1_{\rm loc}(\RN)$ (see also \cite[Remark~2.2]{T-K}).

We next provide \emph{actions of dislocations} defined on homogeneous Sobolev spaces. 

\begin{definition}\label{definition;dislocation of Sobolev}
Let $q \in \R$.  
Define a group action on $L^1_{\rm loc}(\RN)$ as follows: 
\begin{align*}
&
G[\RN \!, \, \Z; q] 
\\
&\quad 
\coloneqq 
\qty{
g[y,j;q] : L^1_{\rm loc}(\RN) \to L^1_{\rm loc}(\RN)
\left| \ 
\begin{aligned}
&
g[y,j;q ] u(\cdot) \coloneqq 
2^{jN/q} u(2^j(\cdot - y)), 
\\ 
&
u \in L^1_{\rm loc}(\RN), \ y \in \RN, \ j\in\Z 
\end{aligned}
\right. 
}
\end{align*}
When $q=\ppm,$ this is a group of bijective isometries on the Sobolev space $\Wdmp$ defined as above:
\begin{equation*}
\|g[y,j;\ppm] u\|_{\Wdmp} = \|u\|_{\Wdmp},  \quad u\in\Wdmp, \ y\in\RN, \  j\in\Z.
\end{equation*}
Moreover, inverse mappings 
are given as follows.
For $g[y,j;\ppm], g[z,k;\ppm]\in G[\RN \!, \, \Z;\ppm]$,  $u\in \Wdmp \ (1\le p < N/m)$, 
\begin{align*}
g[y,j;\ppm]^{-1}u(\cdot)&=2^{-jN/\ppm}u( 2^{-j} \cdot+y)
=g[-2^jy,-j;\ppm]u(\cdot), \\
g[z,k;\ppm]^{-1}g[y,j;\ppm] u(\cdot) 
&=2^{(j-k)N/\ppm}u( 2^{j-k}(\cdot -2^{k}(y-z))) \\
&=g[2^k(y-z), j-k;\ppm]u(\cdot).
\end{align*}
\end{definition}

\begin{remark}
\rm 
One can always replace $G[\RN,\Z;\ppm]$ with $G[\RN,\R;\ppm]$ in all of results in this paper. 
\end{remark}

The following lemma characterizes the operator-weak convergence of dislocations on Sobolev spaces. 

\begin{lemma}\label{lemma;Sobolev dsl weak conv}
Let $1<p<N/m$, $m,N \in \N$, and let $(g[y_n,j_n;\ppm])$ and $(g[z_n,k_n;\ppm])$ be 
sequences in $G[\RN \!, \Z;\ppm]$. As bounded operators 
on $\Wdmp$, the  operator-weak convergence 
$g[y_n,j_n;\ppm]\rightharpoonup 0$ is 
equivalent to $|y_n|+|j_n|\toinfty$,  and 
$g[z_n,k_n;\ppm]^{-1}g[y_n,j_n;\ppm] \rightharpoonup 0$ 
is equivalent to $|j_n-k_n|+2^{k_n}|y_n-z_n| \toinfty$.
\end{lemma}

\begin{proof}
See~\cite[Lemmas~3.1  and~5.1]{T-K}.
\end{proof}

Under these settings, homogeneous Sobolev spaces form dislocation spaces.

\begin{proposition}\label{prop;Wdmp dislocation space}
Let $p\in ]1,N/m[$ and let $m, N\in \N$ with $m<N$. 
Then  $\qty( \Wdmp, G[\RN \!, \, \Z;\ppm])$ is a dislocation space.
\end{proposition}

\begin{proof}
We shall show that $G[\RN \!,\Z;\ppm]$ is a dislocation group.
For a sequence $(g[y_n,j_n;\ppm])$ and $g[y,j;\ppm]$ in $G[\RN \!,\Z;\ppm]$, 
it is easily checked that $g[y_n,j_n;\ppm] \to g[y,j;\ppm]$ operator-strongly if and only if $|j_n-j|+2^j |y_n-y| \to 0$ as $n\toinfty$, and thus, one sees that $g[y_n,j_n;\ppm] \to g[y,j;\ppm]$ operator-strongly if and only if $g[y_n,j_n;\ppm]^{-1} \to g[y,j;\ppm]^{-1}$ operator-strongly. 
Let $(g[y_n,j_n;\ppm])$ be a bounded sequence in $G[\RN \!,\Z;\ppm]$ 
such that $g[y_n,j_n;\ppm] \not\rightharpoonup 0$,  which equivalently 
means that 
$(y_n)$ and $(j_n)$ are both bounded  
due to \cref{lemma;Sobolev dsl weak conv}.
Then there exist a subsequence, still denoted by $n$, $y\in\RN$ and $j\in\Z$ 
such that $y_n \to y, \  j_n =j \  (n\toinfty)$.
Hence it follows that $g[y_n,j_n;\ppm] \to g[y,j;\ppm]$ operator-strongly. 
Finally, we shall show that 
$(g[y_n,j_n;\ppm])^* \to (g[y,j;\ppm])^*$ operator-strongly 
(recall (ii) of \cref{rem;DslSp}).
Indeed, let $u\in\Wdmp$ and let $\phi\in [ \Wdmp ]^*$.
Take functions $\phi_{\alpha}\in L^{p'}(\RN), \ |\alpha|= m$, 
given by \cref{lemma;dual Wdmp}.
It follows from the change of  variables and the H\"older inequality that 
\begin{align*}
&\norm{ (g[y_n,j_n;\ppm])^*\phi-(g[y,j;\ppm])^*\phi }_{[\Wdmp]^*}  \\
&=
\sup_{\substack{ u\in\Wdmp \\ \norm{u}_{\Wdrn}=1 } }
 \qty|\la  (g[y_n,j_n;\ppm])^*\phi-(g[y,j;\ppm])^*\phi, u\ra | \\
&=
\sup_{\substack{ u\in\Wdmp \\ \norm{u}_{\Wdrn}=1 } }
 \qty|\la  \phi, g[y_n,j_n;\ppm]u-g[y,j;\ppm]u\ra | \\
&=
\sup_{\substack{ u\in\Wdmp \\ \norm{u}_{\Wdrn}=1 } }
\qty|\sum_{|\alpha|=m}
\int_{\RN} \phi_\alpha \d^\alpha \qty(g[y_n,j_n;\ppm]u-g[y,j;\ppm]u) \,\dd x| \\
&=
\sup_{\substack{ u\in\Wdmp \\ \norm{u}_{\Wdrn}=1 } }
\Biggl|\sum_{|\alpha|=m}
\int_{\RN} \bigl( 2^{-j_n\frac{N}{p'}}\phi_\alpha(2^{-j_n}x+y_n) 
- 2^{-j\frac{N}{p'}}\phi_\alpha(2^{-j}x+y) \bigr) \d^\alpha u \,\dd x\Biggr| \\
&=
\sup_{\substack{ u\in\Wdmp\\ \norm{u}_{\Wdrn}=1 } }
\qty|\sum_{|\alpha|=m} \int_{\RN} \qty(
g[y_n, j_n; p']^{-1} \phi_{\alpha} - g[y, j; p']^{-1} \phi_{\alpha}
) \d^\alpha u \,\dd x| \\
&\le 
\sup_{\substack{ u\in\Wdmp\\ \norm{u}_{\Wdrn}=1 } }
\qty( \sum_{|\alpha|=m} 
\norm{
g[y_n, j_n; p']^{-1} \phi_{\alpha} 
- g[y, j; p']^{-1} \phi_{\alpha}
}_{L^{p'}(\RN)} ) 
\|u\|_{\Wdmp}  \\
&= 
\sum_{|\alpha|=m} 
\norm{ g[y_n, j_n; p']^{-1} \phi_{\alpha} - g[y, j; p']^{-1} \phi_{\alpha}
}_{L^{p'}(\RN)}.
\end{align*}
The last term is convergent to zero as $n\toinfty$  
since $y_n \to y, \ j_n \to j$ implies that 
$\| g[y_n, j_n; p']^{-1} u - g[y, j; p']^{-1} u \|_{L^{p'}(\RN)} \to 0 $ for all $u \in L^{p'}(\RN)$. 
Hence one gets 
$(g[y_n,j_n;\ppm])^*\phi-(g[y,j;\ppm])^*\phi \to 0$ strongly in $[\Wdmp]^*$, whence follows $(g[y_n,j_n;\ppm])^* \to (g[y,j;\ppm])^* $ operator-strongly.
This completes the proof.
\end{proof}

Notice that the homogeneous Sobolev space $\Wdmp$ with the above dislocation group $G[\RN,\Z;\ppm]$ forms a dislocation space irrelevant to the equivalent Sobolev norms. 
We finally exhibit an example of $G$-completely continuous embedding of the homogeneous Sobolev space into the critical Lebesgue space.
The definition of homogeneous Besov spaces will be described in \cref{section:preliminary} below for the sake of the reader's convenience.

\begin{lemma}\label{lemma;G-complete continuity of Wdmp}
Let $\qty(\Wdmp,G[\RN \!,\Z;\ppm])$ be as in \cref{prop;Wdmp dislocation space}.
Then the continuous embedding 
$\Wdmp \hookrightarrow L^{\ppm}(\RN)$
is $G[\RN \!,  \Z;\ppm]$-completely continuous.
Indeed, it holds that 
\begin{equation*}
\|u\|_{L^{\ppm}(\RN)} \le C \|u\|_{\Wdmp}^{p/\ppm}
\|u\|_{\dot{B}^{-N/\ppm}_{\infty,\infty}(\RN)}^{1-p/\ppm}
\end{equation*}
for all $u\in\Wdmp$,  
and that for a sequence $(u_n)$,  
\begin{equation*}
\|u_n\|_{\dot{B}^{-N/\ppm}_{\infty,\infty}(\RN)} \to 0
\end{equation*}
whenever $u_n \to 0$ $G[\RN \!,\Z;\ppm]$-weakly 
in $\Wdmp$.
\end{lemma}

\begin{proof}
See~\cite[Theorems~1.1.9, and~3.2.1]{T-01}.
\end{proof}

This lemma is generalized as follows: 

\begin{lemma}\label{lemma;G-complete continuity of Wdmp2}
Let $\qty( \Wdmp,G[\RN \!,\Z;\ppm])$ be as in \cref{prop;Wdmp dislocation space}, and let $k\in\NN$ with $k<m$ (yielding $p^*_{m-k}<\ppm$). 
Then the continuous embedding 
$\Wdmp \hookrightarrow W^{k,p^*_{m-k}}(\RN)$ 
is $G[\RN \!,  \Z;\ppm]$-completely continuous.
Moreover, it holds that 
\begin{equation*}
\|u\|_{\dot{W}^{k,p^*_{m-k}}(\RN)} \le C \|u\|_{\Wdmp}^{p/p^*_{m-k}}
\|u\|_{\dot{B}^{-N/p^*_{m-k}}_{\infty,\infty}(\RN)}^{1-p/p^*_{m-k}}
\end{equation*}
for all $u\in\Wdmp$,  
and that for a sequence $(u_n)$,  
\begin{equation*}
\|u_n\|_{\dot{B}^{-N/p^*_{m-k}}_{\infty,\infty}(\RN)} \to 0
\end{equation*}
whenever $u_n \to 0$ $G[\RN \!,\Z;\ppm]$-weakly 
in $\Wdmp$.
\end{lemma}

\begin{proof}
See~\cite[Corollary~3.2.2]{T-01}.
\end{proof}

\subsection{Main theorem}
Now we move on to a fundamental theorem of profile decomposition in the \emph{homogeneous} Sobolev space $\Wdmp$. 
Namely, every bounded sequence in $\Wdmp$ has a fine profile decomposition in the sense that  the residual term is vanishing in $\dot{W}^{k,p^*_{m-k}}(\RN)$.

\begin{theorem}\label{theorem;profile-decomp. in Wdmp}
Let $m,N \in \N$, let $1 < p< N/m$ and let $(u_n)$ be a bounded sequence in $\Wdmp$.
Then there exist $\Lambda \in \N\cup\{0,+\infty\}$, 
a subsequence $(N(n)) \preceq (n) $,
$\bbl{w}{l} \in \Wdmp \ (l \in \NN^{<\Lambda +1})$,  
$(\dsl{y}{l}{N(n)},\dsl{j}{l}{N(n)} ) \in \RN\times\Z \ (l \in \NN^{<\Lambda +1}, \ n \in \Z_{ \ge l})$, 
and  residual terms 
$r^L_{N(n)}\in\Wdmp$ $(L\in\NN^{<\Lambda+1}, \ n\in\Z_{\ge L})$
with the relation of a double-suffix profile decomposition 
\begin{equation}\notag 
    u_{N(n)} = \sum_{l=0}^L 
    2^{\dsl{j}{l}{N(n)}N/\ppm} \bbl{w}{l} \qty(2^{\dsl{j}{l}{N(n)}} (\cdot - \dsl{y}{l}{N(n)}))+r^L_{N(n)} , 
    \quad L \in \NN^{<\Lambda+1}, \ n \in \Z_{\ge L}, 
\end{equation}
such that the following hold:
\begin{enumerate}
\setlength{\itemsep}{0mm}
\setlength{\parskip}{0mm}

    \item $\dsl{y}{0}{N(n)}=0, \s \dsl{j}{0}{N(n)} =0 \ (n \ge 0), \quad \bbl{w}{l}\neq 0 \ (1 \le l \in \NN^{<\Lambda +1})$.
    \item For each $l \neq  k \in \NN^{<\Lambda +1}$, it holds that 
    \begin{equation}\label{202106240010}
    |\dsl{j}{l}{N(n)}-\dsl{j}{k}{N(n)}|+
    2^{\dsl{j}{k}{N(n)}}
    |\dsl{y}{l}{N(n)}-\dsl{y}{k}{N(n)}|
    \toinfty 
    \end{equation}
    as $n \toinfty$.  
    Furthermore, one can assume without loss of generality that, for $l\neq k\in\NN^{<\Lambda+1}$, one and only one of the following three cases occurs: 
    \begin{enumerate}
        \item $\dsl{j}{l}{N(n)} - \dsl{j}{k}{N(n)} \to +\infty$ $(n\toinfty);$ 
        \item $\dsl{j}{l}{N(n)} - \dsl{j}{k}{N(n)} \to -\infty$ $(n\toinfty);$ 
        \item $(\dsl{j}{l}{N(n)} - \dsl{j}{k}{N(n)}) $ is convergent.  
    \end{enumerate}
    One can also assume without loss of generality that, for $l\neq k\in\NN^{<\Lambda+1}$, one and only one of the following two cases occurs: 
    \begin{enumerate}
    \setcounter{enumii}{3}
        \item $2^{\dsl{j}{k}{N(n)}} |\dsl{y}{l}{N(n)}-\dsl{y}{k}{N(n)}| \to\infty \ (n\to\infty);$ 
        \item $(2^{\dsl{j}{k}{N(n)}} (\dsl{y}{l}{N(n)}-\dsl{y}{k}{N(n)}))$ is convergent.  
    \end{enumerate}
    But the cases $(c)$ and $(e)$ do not occur simultaneously due to~\eqref{202106240010}.
    \item 
    $2^{-\dsl{j}{l}{N(n)}N/\ppm} u_{N(n)} \qty( 2^{-\dsl{j}{l}{N(n)}} \cdot + \dsl{y}{l}{N(n)}) \to \bbl{w}{l}$ as $n\toinfty$ 
    weakly in  $\Wdmp$ and a.e.\  on $\RN$ $(l \in \NN^{<\Lambda +1})$. 
    \item For $k\in\NN^{<\Lambda+1}$,  
    \begin{equation*}
    2^{-\dsl{j}{k}{N(n)} N/\ppm } r^{L}_{N(n)} \qty(2^{-\dsl{j}{k}{N(n)}}\cdot + \dsl{y}{k}{N(n)}) \to 
    \begin{cases}
    0 & \mbox{if}\ k=0,\ldots,L, \\
    \bbl{w}{k} & \mbox{if}\ k\ge L+1,
    \end{cases}
    \end{equation*}
    weakly in $\Wdrn$ and a.e.\ on $\RN$ as $n\toinfty$. 
\end{enumerate} 
Moreover, the following energy relations hold:
\begin{align}
&\label{eq;energy estim Wdmp}
\varlimsup_{n\toinfty}\|u_{N(n)} \|_{\Wdmp}^p \\ 
&\notag \qquad \ge \sum_{l=0}^\Lambda \|\bbl{w}{l}\|_{\Wdmp}^p 
+ \varlimsup_{L\to \Lambda} \varlimsup_{R\toinfty} \varlimsup_{n\toinfty}
\| r^L_{N(n)} \|^p_{\dot{W}^{m,p}(\RN \setminus \Bnrl )  },   
\\
&\lim_{L\to \Lambda}\sup_{\phi\in U}\varlimsup_{n\toinfty}\sup_{y\in\RN \!, \,  j\in\Z}
\qty|\la \phi, 2^{-jN/\ppm} r^L_{N(n)}(2^{-j} \cdot + y) \ra |=0, \label{eq;Ishiwata condition Wdmp}  \\ 
&\varlimsup_{L\to \Lambda}\varlimsup_{n\toinfty}\|r^{L}_{N(n)}\|_{\Wdmp} \le 2 \varlimsup_{n\toinfty}\|u_{N(n)} \|_{\Wdmp}<+ \infty, \label{eq;residue bounded Wdmp}\\
&    \lim_{L\to \Lambda}\varlimsup_{n \toinfty} \|r^L_{N(n)} \|_{L^{\ppm}(\RN)}=0, \label{residue vanish Wdmp} \\
&    \lim_{L\to \Lambda}\varlimsup_{n \toinfty} \|r^L_{N(n)} \|_{\dot{W}^{k,p^*_{m-k}}(\RN)}=0,\quad k\in\NN, \ k<m, \label{residue vanish Wdmp2} 
\end{align}
where $U \coloneqq B_{[\Wdmp]^*}(1)$ and  $\Bnrl \coloneqq \bigcup_{l=0}^L 
B( \dsl{y}{l}{N(n)},  2^{-\dsl{j}{l}{N(n)}} R).$
\end{theorem}

\begin{remark} 
\begin{enumerate}
\item 
The meaning of each assertion above is as follows:  
\eqref{eq;energy estim Wdmp} indicates that the sum of ($p$-powered) Sobolev norms of all profiles is bounded by  
$\varlimsup_{n\toinfty} \|u_{N(n)} \|_{\Wmp}$, 
which is, so to speak, an \emph{energy estimate} or an \emph{energy decomposition}; \eqref{eq;Ishiwata condition Wdmp} implies the completion of performing the profile decomposition and called the exactness condition; \eqref{eq;residue bounded Wdmp} ensures that the residual term does not diverge as the number of subtracted dislocated profiles increases; \eqref{residue vanish Wdmp} and~\eqref{residue vanish Wdmp2} show that lower order derivatives of the residual term are vanishing strongly in appropriate Lebesgue or Sobolev spaces.

\item When one considers profile decomposition in Sobolev spaces $\dot{W}^{m,2}(\RN)$, which is a Hilbert space, one should employ the profile decomposition theorem in Hilbert spaces described in~\cite[Theorem~2.9]{Okumura2} rather than the above theorem, since it provides a more precise energy decomposition for the Sobolev norm, i.e., one obtains 
\begin{equation*}
    \varlimsup_{n\toinfty} \|u_{N(n)}\|_{\dot{W}^{m,2}(\RN)}^2
    = \sum_{l =0}^\Lambda \|\bbl{w}{l}\|_{\dot{W}^{m,2}(\RN)}^2+ 
    \lim_{L\to \Lambda} \varlimsup_{n\toinfty} \|r^{L}_{N(n)}\|_{\dot{W}^{m,2}(\RN)}^2. 
\end{equation*}

\item 
Regarding the energy estimate~\eqref{eq;energy estim Wdmp}, it is noteworthy that this is a sharper version than other authors' energy decompositions because other types of one do not include the ``residual part'' as in~\eqref{eq;energy estim Wdmp}.

\item 
Relations~\eqref{eq;Ishiwata condition Wdmp}--\eqref{residue vanish Wdmp2} are also obtained in the previous researches, e.g., \cite{Gerard,Jaffard,Solimini,T-01,T-K}. 

\end{enumerate}
\end{remark}


Once the number of nontrivial profiles turns out to be finite, then the assertions above become simpler. 

\begin{corollary}
Suppose that the same conditions as in \cref{theorem;profile-decomp. in Wdmp} are satisfied. 
In addition, assume that $\Lambda$ is finite, i.e., 
the number of nontrivial profiles is finite. 
Then regarding the final residual term given by 
\begin{equation*}
r^\Lambda_{N(n)} = u_{N(n)} -\sum_{l=0}^\Lambda 
    2^{\dsl{j}{l}{N(n)}N/\ppm} \bbl{w}{l} \qty(2^{\dsl{j}{l}{N(n)}} (\cdot - \dsl{y}{l}{N(n)})), 
    \quad n \ge \Lambda, 
\end{equation*}
the relations~\eqref{eq;Ishiwata condition Wdmp}--\eqref{residue vanish Wdmp} turn to 
\begin{align}
&\notag 
\varlimsup_{n\toinfty}\|u_{N(n)} \|_{\Wdmp}^p \\ 
& \notag \qquad 
\ge \sum_{l=0}^\Lambda \|\bbl{w}{l}\|_{\Wdmp}^p 
    +  \varlimsup_{R\toinfty} \varlimsup_{n\toinfty}
    \| r^\Lambda_{N(n)} \|^p_{\dot{W}^{m,p}(\RN \setminus \mathcal{B}_{n,R,\Lambda} )  },   
   \\
&\lim_{n\toinfty}\sup_{y\in\RN \!, \, j\in\Z}
    \qty|\la \phi, 2^{-jN/\ppm} r^\Lambda_{N(n)}(2^{-j} \cdot + y) \ra |=0,
\quad  \phi \in U, \label{residue G weak conv Wdmp}\\
&\varlimsup_{n\toinfty}\|r^\Lambda_{N(n)}\|_{\Wdmp}
\le 2 \varlimsup_{n\toinfty} \|u_{N(n)}\|_{\Wdmp}, \notag \\
&\lim_{n\toinfty}\|r^\Lambda_{N(n)}\|_{L^{\ppm}(\RN)}=0, \notag \\ 
&\lim_{n\toinfty}\|r^\Lambda_{N(n)}\|_{\dot{W}^{k,p^*_{m-k}}(\RN)}=0,\quad k\in\NN, \ k<m, \notag 
\end{align}
where $U \coloneqq B_{[\Wdmp]^*}(1)$ and $\mathcal{B}_{n,R,\Lambda} \coloneqq  \bigcup_{l=0}^\Lambda B( \dsl{y}{l}{N(n)},  2^{-\dsl{j}{l}{N(n)}} R)$.
In particular, the relation~\eqref{residue G weak conv Wdmp} implies that the final residual term is $G[\RN \!, \Z;\ppm]$-weakly convergent to zero in $\Wdmp$. 
\end{corollary}

The profile decomposition remains true in $\Wdmp$ equipped with an equivalent norm that is often used.

\begin{corollary}\label{corollary;another homogeneous norm}
Under the same conditions as in \cref{theorem;profile-decomp. in Wdmp},  
all of assertions in the theorem hold true even when the norm 
of $\dot{W}^{1,p}(\RN)$ is replaced by 
$$
\|u\|_{\dot{W}^{1,p}(\RN)}\coloneqq 
\l( \int_{\RN} |\nabla u|^p \,\dd x \r)^{1/p}\!, \quad u\in \dot{W}^{1,p}(\RN),
$$
where $|\nabla u|=\sqrt{|\frac{\d u}{\d x_1}|^2 + \cdots + |\frac{\d u}{\d x_N}|^2}$.
\end{corollary}

\begin{proof}
The corollary can be proved in the same way as the proof of \cref{theorem;profile-decomp. in Wdmp}.
\end{proof}

\subsection{Proof of \cref{theorem;profile-decomp. in Wdmp}}
We shall employ \cref{theorem;ProDeco in X} to obtain profiles, and it remains to show the estimates for the Sobolev and Lebesgue norms. 
In what follows, for the sake of simplicity, we often denote 
$g[y,j;\ppm],  g[\dsl{y}{l}{n},\dsl{j}{l}{n};\ppm] \in G[\RN,\Z;\ppm]$ by $g, \dsl{g}{l}{n}$, respectively, for short.
%
Throughout this proof, when a sequence $(u_n)$ converges to $u$ weakly in $\Wdmp$ and almost everywhere on $\RN$, we denote it by 
\begin{equation*}
u_n \to u \quad \mbox{weakly and a.e.}
\end{equation*}
for short. 
\begin{step}[Finding profiles]
\cref{theorem;ProDeco in X} leads us to the existence of profile elements $(\bbl{w}{l}, \dsl{y}{l}{N(n)}, \dsl{j}{l}{N(n)}, \Lambda ) \in \Wdmp \times \RN \times \Z\times (\N\cup\{0,\infty\})$  $(l \in \NN^{<\Lambda+1}, \  n \in \Z_{\ge l})$ satisfying all assertions in \cref{theorem;ProDeco in X}. As in the proof of \cref{theorem;ProDeco in X} in~\cite{Okumura2}, each profile $\bbl{w}{l}$ is obtained as the weak limit of $\dsl{g}{l}{i(l,n)}^{-1} u_{i(l,n)}$, where $(i(l,n))$ denotes the $l$-th subsequence for which the $l$-th profile and the $l$-th dislocations are considered in an iterative process. However,  with the help of the Sobolev compact embeddings, the weak convergence $u_n \to u$ in $\Wdmp$ also leads us to the pointwise convergence $u_n \to u$ a.e.\  on $\RN$ up to a subsequence.
Moreover, for the $L$-th subsequence $(i(L,n))$ $(L \in \NN^{<\Lambda+1})$, 
one has 
\begin{align}
\label{202106240020}
| \dsl{j}{l}{i(L,n)} - \dsl{j}{k}{i(L,n)} | + 2^{\dsl{j}{k}{i(L,n)}} |\dsl{y}{l}{i(L,n)} - \dsl{y}{k}{i(L,n)} | \toinfty \\
\notag (n\toinfty, \ 0\le l \neq k \le L).
\end{align}

Thus, on a subsequence, still denoted by $i(L,n)$, one can always assume that,  for each $0 \le l \neq k \le L$,  one and only one of the following three conditions occurs: 
\begin{enumerate}
\setlength{\itemsep}{0mm}
\setlength{\parskip}{0mm}

    \item $\dsl{j}{l}{i(L,n)} - \dsl{j}{k}{i(L,n)} \to + \infty$ $(n\toinfty);$
    \item $\dsl{j}{l}{i(L,n)} - \dsl{j}{k}{i(L,n)} \to - \infty$ $(n\toinfty);$
    \item $(\dsl{j}{l}{i(L,n)} - \dsl{j}{k}{i(L,n)} )$ is convergent in $\Z$ as $n\toinfty$,
\end{enumerate}
and one can also assume that, for $0\le l\neq k\le L$, 
one and only one of the following two cases occurs:
\begin{enumerate}
\setlength{\itemsep}{0mm}
\setlength{\parskip}{0mm}

\setcounter{enumi}{3}
    \item $2^{\dsl{j}{k}{i(L,n)}} |\dsl{y}{l}{i(L,n)} - \dsl{y}{k}{i(L,n)} | \toinfty \ (n\to\infty);$
    \item $(2^{\dsl{j}{k}{i(L,n)}} (\dsl{y}{l}{i(L,n)} - \dsl{y}{k}{i(L,n)} ))$ is convergent in $\RN$ as $n\toinfty$. 
\end{enumerate}
But~(iii) and~(v) do not occur simultaneously due to~\eqref{202106240020}.
Hence by extracting additional subsequences in each step of the proof,  one can obtain both the weak convergence and the pointwise convergence for profiles and residual terms, and also get the above trichotomy and dichotomy regarding dislocations. 

Therefore, according to \cref{theorem;ProDeco in X} we have:   
\begin{align}
\dsl{y}{0}{N(n)}=0, \ \dsl{j}{l}{N(n)}=0 \quad  (n\ge 0), \notag 
\\ 
\notag 
\bbl{w}{l}\neq 0 \quad (1 \le l \in \NN^{<\Lambda +1}),
\\
 \label{eq;500506}
\l| \dsl{j}{l}{N(n)} - \dsl{j}{k}{N(n)} \r| + 2^{\dsl{j}{k}{N(n)}}\l| \dsl{y}{l}{N(n)} - \dsl{y}{k}{N(n)} \r| \toinfty  
\\
\notag \hspace{4cm}
(n\toinfty, \ k \neq l \in \NN^{<\Lambda +1}),  
\\
\label{202105210050} 
2^{-\dsl{j}{l}{N(n)}N/\ppm} u_{N(n)} \qty(2^{-\dsl{j}{l}{N(n)}}\cdot+\dsl{y}{l}{N(n)}) \to \bbl{w}{l}  
\\
\notag  \hspace{4cm}
\mbox{weakly and a.e.} \ 
(l \in \NN^{<\Lambda +1}). 
\end{align}
Set the residual term by 
\begin{align*}
r^{L}_{N(n)} \coloneqq u_{N(n)}-\sum_{l=0}^L 2^{\dsl{j}{l}{N(n)}N/\ppm} \bbl{w}{l} \qty(2^{\dsl{j}{l}{N(n)}} (\cdot- \dsl{y}{l}{N(n)})), \quad 
L \in \NN^{<\Lambda +1}, \ n \in \Z_{\ge L}.
\end{align*}
Then the triangle inequality yields 
\begin{equation}
 \sup_{n\ge 0} \|r^L_{N(n)} \|_{\Wdmp}
\le 
\sup_{n\ge 0} \|u_{N(n)}\|_{\Wdmp} 
+ \sum_{l=0}^L \|\bbl{w}{l}\|_{\Wdmp}.   \label{202105210070}
\end{equation}
Also, \cref{theorem;ProDeco in X} implies 
\begin{align*}
2^{-\dsl{j}{l}{N(n)}N/\ppm}r^{L}_{N(n)} \qty(2^{-\dsl{j}{l}{N(n)}} \cdot+ \dsl{y}{l}{N(n)})  \to 0
\quad  \mbox{weakly and a.e.}    \\ 
(n\to \infty, \ 0  \le l \le  L\in \NN^{<\Lambda +1}), \\
2^{-\dsl{j}{L}{N(n)}N/\ppm}r^{L-1}_{N(n)} \qty(2^{-\dsl{j}{L}{N(n)}} \cdot+ \dsl{y}{L}{N(n)})  \to \bbl{w}{L} 
\quad \mbox{weakly and a.e.}  \\
(n\to \infty, \ 1  \le L\in \NN^{<\Lambda +1}).
\end{align*}
Moreover, one can assume that, for $l \neq k \in \NN^{<\Lambda +1}$, one and only one of the following three cases holds: 
\begin{enumerate}
\setlength{\itemsep}{0mm}
\setlength{\parskip}{0mm}

    \item $\dsl{j}{l}{N(n)} - \dsl{j}{k}{N(n)} \to + \infty$ $(n\toinfty);$
    \item $\dsl{j}{l}{N(n)} - \dsl{j}{k}{N(n)} \to - \infty$ $(n\toinfty);$
    \item $\dsl{j}{l}{N(n)} - \dsl{j}{k}{N(n)}$ is convergent in $\Z$  as  $n\toinfty$, 
\end{enumerate}
and one can also assume that, for $l\neq k\in\NN^{<\Lambda+1}$, 
one and only one of the following two cases holds:
\begin{enumerate}
\setcounter{enumi}{3}
\setlength{\itemsep}{0mm}
\setlength{\parskip}{0mm}

    \item $2^{\dsl{j}{k}{N(n)}} | \dsl{y}{l}{N(n)} - \dsl{y}{k}{N(n)} |\to\infty$ $(n\to\infty);$
    \item $(2^{\dsl{j}{k}{N(n)}} ( \dsl{y}{l}{N(n)} - \dsl{y}{k}{N(n)} ))$ is convergent in $\RN$ as $n\toinfty$.
\end{enumerate}
But the cases~(iii) and~(v) do not occur simultaneously due to~\eqref{eq;500506}. 
The remaining assertions~\eqref{eq;energy estim Wdmp}--\eqref{residue vanish Wdmp}
will be proved in the following five lemmas. 
\end{step}

\begin{step}[Energy decomposition and exactness condition]
\begin{lemma}[Estimate~\eqref{eq;energy estim Wdmp} 
and the exactness condition~\eqref{eq;Ishiwata condition Wdmp}]
\label{lemma;energy estim Wdmp}
It holds that 
\begin{align}
&\label{eq;20019711}
\varlimsup_{n\toinfty}\|u_{N(n)}\|_{\Wdmp}^p 
\\
&\notag  \qquad 
\ge 
\sum_{l=0}^\Lambda \|\bbl{w}{l}\|_{\Wdmp}^p 
+ \varlimsup_{L\to \Lambda} \varlimsup_{R\toinfty} \varlimsup_{n\toinfty}
\| r^L_{N(n)} \|^p_{\dot{W}(\RN\setminus \Bnrl) },
\end{align}
where $\Bnrl \coloneqq \bigcup_{l=0}^L B(\dsl{y}{l}{N(n)}, 2^{ - \dsl{j}{l}{N(n)} } R  ). $
Furthermore, it follows that 
\begin{equation}\label{202104040040}
\lim_{L\to \Lambda}\sup_{\phi\in U}\varlimsup_{n\toinfty}\sup_{y\in \RN \!, \,  j \in \Z}
\qty|\la  \phi, 2^{-jN/\ppm}r^{L}_{N(n)} (2^{-j} \cdot+ y)  \ra |=0,  
\end{equation}
where $U \coloneqq B_{[\Wdmp]^*}(1)$. 
\end{lemma}

\begin{proof}
Once one has proved the energy estimate~\eqref{eq;20019711}, then the latter condition~\eqref{202104040040} will be readily obtained due to \cref{theorem;ProDeco in X}. 
Hence we shall only prove~\eqref{eq;20019711}. 
Since we have the trichotomy on dislocations $\dsl{j}{l}{N(n)}$, we  can define the following disjoint sets: for $0 \le l \le L \in \NN^{<\Lambda+1}$, 
\begin{align*}
\Jlp &\coloneqq \qty{ l' \in \{0,\ldots, L\}; \ 
\dsl{j}{l'}{N(n)} -\dsl{j}{l}{N(n)} \to +\infty 
\ \mbox{as} \ n\toinfty }, \\
\Jln &\coloneqq \qty{ l' \in \{0,\ldots, L\}; \ 
\dsl{j}{l'}{N(n)} -\dsl{j}{l}{N(n)} \to -\infty
\ \mbox{as} \ n\toinfty }, \\
\Jlz &\coloneqq \qty{ l' \in \{0,\ldots, L\}; \ 
( \dsl{j}{l'}{N(n)} -\dsl{j}{l}{N(n)}  ) \  \mbox{is convergent} },\\
\Jlp &\sqcup \Jln \sqcup \Jlz = \{0,\ldots,L\}, 
\end{align*}
where $\sqcup$ denotes a disjoint union of sets.
Roughly speaking, $\Jlp$ stands for the set of profile numbers up to $L$ whose profiles concentrate faster than the $l$-th profile $\bbl{w}{l}$; 
$\Jln$ for the set of profile numbers up to $L$ whose profiles concentrate slower than the $l$-th profile $\bbl{w}{l}$;
$\Jlz$ for the set of profile numbers up to $L$ whose profiles concentrate at the same speed as the $l$-th profile $\bbl{w}{l}$.

Let $\chi$ be a characteristic function supported on the unit ball $B(0,1)$ and 
set $\chi_R(\cdot) \coloneqq \chi(R^{-1}\cdot) = \chi_{B(0,R)}$. 
Also we define the scaling action $\sigma[y,j]$ 
for $y \in \RN $ and $ j \in \Z$ by 
\begin{equation*}
(\sigma[y,j]u)(\cdot) \coloneqq 
u ( 2^j (\cdot -y)),  \quad u \in L^1_{\rm loc}(\RN).
\end{equation*}
It is easily checked that 
\begin{align*}
g[y,j;\ppm] u 
&
=2^{jN/\ppm} u (2^{j}(\cdot -y)) 
=2^{ j N/\ppm} \sigma[y,j] u,
\\
\d^\alpha  (g[y,j;\ppm]u)
&
=2^{ jN/\ppm  } 2^{mj} \sigma[y,j](\d^\alpha u)
\\
&=2^{jN/p} \sigma[y,j] (\d^\alpha u)
=g[y,j;p](\d^\alpha u),
\end{align*}
where $\alpha \in (\NN)^N$ with $|\alpha|=m$. 
For dislocations $(\dsl{y}{l}{n}, \dsl{j}{l}{n}) \in \RN \times \Z$, we write 
\begin{equation*}
\dsl{\sigma}{l}{n} u  \coloneqq 
\sigma[\dsl{y}{l}{n}, \dsl{j}{l}{n}] u 
\end{equation*}
for short. 
Then the characteristic function supported on $B(\dsl{y}{l}{N(n)}, 2^{-\dsl{j}{l}{N(n)} } R )$ can be written by 
$\dsl{\sigma}{l}{N(n)} \chi_R$.

Now we provide the following  identity: 
for sufficiently large $n \gg 1$, 
\begin{equation}\label{202105020090}
    1 \equiv
     \sum_{l=0}^L \dsl{\sigma}{l}{N(n)} 
 \chi_R \prod_{l' \in \Jlp } (1- \dsl{\sigma}{l'}{N(n)} \chi_R )
  +  \prod_{l=0}^L (1 - \dsl{\sigma}{l}{N(n)} \chi_R ). 
\end{equation}
When $\Jlp = \emptyset$, we always assume $\prod_{l' \in \Jlp } (1-\dsl{\sigma}{l'}{N(n)} \chi_R) \equiv 1$. 
The above identity will be shown by induction on $L$ in \cref{lemma;cut off induction} below. 
The meaning of the identity reads as follows: 
the last term $\prod_{l=0}^L (1 - \dsl{\sigma}{l}{N(n)} \chi_R)$ means the characteristic function supported on $\RN \setminus \Bnrl$, 
and $\dsl{\sigma}{l}{N(n)} 
 \chi_R \prod_{l' \in \Jlp } (1- \dsl{\sigma}{l'}{N(n)} \chi_R )$  
restricts our perspective onto $B(\dsl{y}{l}{N(n)}, 2^{-\dsl{j}{l}{N(n)} } R )$ excluding balls 
$B(\dsl{y}{l'}{N(n)}, 2^{-\dsl{j}{l'}{N(n)} } R )$ 
for $l' \in \mathcal{J}^+_{l,L}$ which stand for the (essential) supports of the profiles  concentrating  faster than the $l$-th profile $\bbl{w}{l}$. 
Roughly speaking, the identity splits $\RN$ into 
disjoint supports of profiles $\bbl{w}{0}, \ldots, \bbl{w}{L}$ and the residual term  $r^L_{N(n)}$.

With these devices, one sees that for large $n$ and 
for a fixed multi-index $\alpha\in(\NN)^N$ with $|\alpha|=m$,  
\begin{align}
\label{202105020010}
\int_{\RN} |\d^\alpha u_{N(n)}|^p \,\dd x 
&
=
\sum_{l=0}^L \int_{\RN} |\d^\alpha u_{N(n)} |^p 
\dsl{\sigma}{l}{N(n)} \chi_R \prod_{l' \in \Jlp } (1- \dsl{\sigma}{l'}{N(n)} \chi_R )
\,\dd x  
\\
&\notag \hspace{3cm}
+
\int_{\RN} |\d^\alpha u_{N(n)} |^p 
  \prod_{l=0}^L (1 - \dsl{\sigma}{l}{N(n)} \chi_R )  \,\dd x 
\\ 
&\notag 
\eqqcolon \sum_{l=0}^L \bbl{J}{l}_1 + J_2. 
\end{align}
Due to the convexity of $|\cdot|^p$, one sees that 
\begin{align}\label{202105020020}
J_2 
& \ge 
\int_{\RN} \l[  |\d^\alpha r^L_{N(n)} |^p +
p|\d^\alpha r^L_{N(n)} |^{p-2} 
\d^\alpha r^L_{N(n)} \d^\alpha \l( 
\sum_{l=0}^L \dsl{g}{l}{N(n)} \bbl{w}{l} \r) \r]  
\\
&\notag \hspace{5cm}
\times \prod_{l=0}^L (1 - \dsl{\sigma}{l}{N(n)} \chi_R ) \,\dd x 
\\ 
&\notag 
= 
\int_{\RN} |\d^\alpha r^L_{N(n)} |^p 
\prod_{l=0}^L  (1 - \dsl{\sigma}{l}{N(n)} \chi_R )  \,\dd x 
\\ 
&\notag 
\qquad + p \sum_{l=0}^L  
\int_{\RN} |\d^\alpha r^L_{N(n)} |^{p-2} 
\d^\alpha r^L_{N(n)} \d^\alpha \l( 
\dsl{g}{l}{N(n)} \bbl{w}{l}  \r)  \\
&\notag 
\hspace{5cm}
\times \prod_{l=0}^L  (1 - \dsl{\sigma}{l}{N(n)} \chi_R ) \,\dd x 
\\ 
&\notag 
\eqqcolon 
J_3
+ p \sum_{l=0}^L \bbl{J}{l}_4. 
\end{align}
As is observed before, one has  
\begin{equation*}
\prod_{l=0}^L  (1 - \dsl{\sigma}{l}{N(n)} \chi_R )
=
\chi_{ \RN \setminus \qty[ \bigcup_{l=0}^L B \qty(\dsl{y}{l}{N(n)}, 2^{ -  \dsl{j}{l}{N(n)} } R  )  ] }
= \chi_{\RN \setminus \Bnrl}, 
\end{equation*}
and hence, we find that 
\begin{align}\label{202105020030}
J_3
=\int_{\RN \setminus \Bnrl}
|\d^\alpha r^L_{N(n)} |^p  \,\dd x.   
\end{align}

From the H\"older inequality, the change of variables and~\eqref{202105210070}, one also gets 
\begin{align}\label{202105020040}
|\bbl{J}{l}_4|
&
\le \| \d^\alpha r^L_{N(n)} \|_{L^p(\RN) }^{p-1}
\times \l( \int_{\RN} |\d^\alpha ( \dsl{g}{l}{N(n)} \bbl{w}{l} )|^p 
\prod_{l=0}^L (1 - \dsl{\sigma}{l}{N(n)} \chi_R ) \,\dd x \r)^{1/p} 
\\ 
&\notag  
\le C_L 
\l( \int_{\RN} |\d^\alpha  \bbl{w}{l} |^p 
 (1 - \chi_{B(0,R)} ) \,\dd x \r)^{1/p}
 \\
&\notag  
= C_L 
\l( \int_{\RN \setminus B(0,R)} 
|\d^\alpha  \bbl{w}{l} |^p 
\,\dd x \r)^{1/p} 
\to 0, \quad \mbox{as} \ R\toinfty. 
\end{align}
From~\eqref{202105020020}--\eqref{202105020040},
one obtains
\begin{equation}
\label{202105090030}
J_2 
\ge 
\int_{\RN \setminus \Bnrl}
|\d^\alpha r^L_{n(n)} |^p  \,\dd x
-C_L \sum_{l=0}^L \l( \int_{\RN \setminus B(0,R)} 
|\d^\alpha  \bbl{w}{l} |^p 
\,\dd x \r)^{1/p}. 
\end{equation}

Now we go back to the first term in the last line of~\eqref{202105020010}.
By changing variables, we get  
\begin{align}\label{202105020050}
\bbl{J}{l}_1
&= \int_{\RN} |\d^\alpha u_{N(n)} |^p 
\dsl{\sigma}{l}{N(n)} \chi_R 
\prod_{l' \in \Jlp } (1- \dsl{\sigma}{l'}{N(n)} \chi_R )  \,\dd x \\
&\notag
= \int_{\RN} |\d^\alpha ( [ \dsl{g}{l}{N(n)}]^{-1} u_{N(n)} ) |^p 
 \chi_{R}
 \times  \prod_{l' \in \Jlp }
 (1- [\dsl{\sigma}{l}{N(n)}]^{-1} \dsl{\sigma}{l'}{N(n)} \chi_R )  \,\dd x, 
\end{align}
where 
$[\dsl{\sigma}{l}{N(n)}]^{-1} \dsl{\sigma}{l'}{N(n)} \chi_R$ 
denotes the characteristic function supported on the ball \\
$B \qty( 2^{\dsl{j}{l}{N(n)}}(\dsl{y}{l'}{N(n)} - \dsl{y}{l}{N(n)}  ), 
2^{- ( \dsl{j}{l'}{N(n)} - \dsl{j}{l}{N(n)} ) } R   )$.
Due to the convexity again, one has 
\begin{align}
\label{202105020060}
 &\int_{\RN} |\d^\alpha ( [ \dsl{g}{l}{N(n)}]^{-1} u_{N(n)} ) |^p 
 \chi_{R} \prod_{l' \in \Jlp }
 (1- [\dsl{\sigma}{l}{N(n)}]^{-1} \dsl{\sigma}{l'}{N(n)} \chi_R )
 \,\dd x \\
 &\notag  \ge 
 \int_{\RN} |\d^\alpha \bbl{w}{l} |^p 
 \chi_{R} \prod_{l' \in \Jlp }
 (1- [\dsl{\sigma}{l}{N(n)}]^{-1} \dsl{\sigma}{l'}{N(n)} \chi_R )
   \,\dd x \\
&\notag  \quad +p
 \int_{\RN} |\d^\alpha \bbl{w}{l} |^{p-2} 
 \d^\alpha \bbl{w}{l} (\d^\alpha ( [ \dsl{g}{l}{N(n)}]^{-1} u_{N(n)} ) - \d^\alpha \bbl{w}{l}  ) \\ 
&\notag \hspace{3cm} 
\times \chi_{R} \prod_{l' \in \Jlp }
 (1- [\dsl{\sigma}{l}{N(n)}]^{-1} \dsl{\sigma}{l'}{N(n)} \chi_R ) \,\dd x. 
\end{align}

Moreover, one can show that 
\begin{align}
& \label{202105110010}
\prod_{l' \in \Jlp } (1-[\dsl{\sigma}{l}{N(n)}]^{-1}
\dsl{\sigma}{l'}{N(n)} \chi_{R}) \to 1 \s \mbox{in} \ L^q(\RN) 
\ ( q \in [1,+\infty[ ),
\\ 
& \label{202105090040}
\int_{\RN} | \d^\alpha \bbl{w}{l} |^p 
 \chi_{R} \prod_{l' \in \Jlp }
 (1-[\dsl{\sigma}{l}{N(n)}]^{-1}
\dsl{\sigma}{l'}{N(n)} \chi_{R})  \,\dd x
\to \int_{\RN} |\d^\alpha \bbl{w}{l} |^p  \chi_{R} \,\dd x, 
\\
&\label{202105090041}
\int_{\RN} |\d^\alpha \bbl{w}{l} |^{p-2} 
 \d^\alpha \bbl{w}{l} (\d^\alpha ( [ \dsl{g}{l}{N(n)}]^{-1} u_{N(n)} ) - \d^\alpha \bbl{w}{l}  )  
 \\
 &\notag \hspace{3cm}
 \times \chi_{R} \prod_{l' \in \Jlp }
 (1-[\dsl{\sigma}{l}{N(n)}]^{-1}
\dsl{\sigma}{l'}{N(n)} \chi_{R}) \,\dd x
 \to 0,
\end{align}
as $n\toinfty$ (these facts will be proved in \cref{lemma;calc in proof of claim E} below).
Hence from~\eqref{202105020050}--\eqref{202105090041},
one gets 
\begin{equation}
\label{202105090060}
\bbl{J}{l}_1
 \ge 
 \int_{\RN} |\d^\alpha \bbl{w}{l} |^p  \chi_R \,\dd x + o(1)
\end{equation}
as $n\toinfty.$
Combining~\eqref{202105020010},~\eqref{202105090030} and~\eqref{202105090060}, 
one obtains that 
\begin{align}
\label{202105090100}
\int_{\RN} |\d^\alpha u_{N(n)}  |^p \,\dd x 
&\ge 
\sum_{l=0}^L  
\int_{\RN} |\d^\alpha \bbl{w}{l} |^p  \chi_{R}  \,\dd x  
+ \int_{\RN \setminus \Bnrl} |\d^\alpha  r^L_{N(n)} |^p \,\dd x
\\ 
&\notag \qquad 
- C_L \sum_{k=0}^L 
\int_{\RN \setminus B(0,R)} | \d^\alpha  \bbl{w}{k}  |^p \,\dd x+ o(1)
\end{align}
as $n \toinfty$. 
Summing up~\eqref{202105090100} over all multi-indices $\alpha$ with $|\alpha|=m$ and 
passing to the limits as $n \toinfty$, $R\toinfty$ and then  $L \to \Lambda$, we see that 
\begin{align*}
&\varlimsup_{n\toinfty}
\| u_{N(n)} \|_{\Wdmp}^p 
\ge 
\sum_{l=0}^\Lambda \| \bbl{w}{l} \|_{\Wdmp}^p  
+ \varlimsup_{L\to \Lambda} \varlimsup_{R\toinfty} \varlimsup_{n\toinfty}
\| r^L_{N(n)} \|_{ \dot{W}^{m,p}(\RN \setminus \Bnrl)  }^p. 
\end{align*}
Hence it remains to prove~\eqref{202105110010}--\eqref{202105090041}.
\end{proof}

\begin{lemma}[Boundedness of the residual term~\eqref{eq;residue bounded Wdmp}]
\label{lemma;residue bounded Wdmp}

It holds that 
\begin{equation*}
    \varlimsup_{L\to \Lambda}\varlimsup_{n\toinfty} \|r^L_{N(n)}\|_{\Wdmp}
    \le 2 \varlimsup_{n\toinfty}\| u_{N(n)}  \|_{\Wdmp} <+\infty.
\end{equation*}
\end{lemma}

\begin{proof}
Set $S^L_{N(n)} \coloneqq \sum_{l=0}^L \dsl{g}{l}{N(n)} \bbl{w}{l}.$
We shall employ the following elementary inequality for the Euclidean norm:
for all $\alpha_l\in\R^d$ $(l=1,\ldots,L, \ d,L\in\N)$, 
\begin{equation*}
\l| \l| \sum_{l=1}^L \alpha_l \r|^p - \sum_{l=1}^L |\alpha_l|^p \r|\le C_L \sum_{l\neq m}|\alpha_l||\alpha_m|^{p-1}
\end{equation*}
for some constant $C_L>0$. 
This inequality will be shown similarly to \cref{lemma;10180009} below.
From this, we see that, for any multi-index $\alpha$ with $|\alpha| = m$, 
\begin{align*}
&\l| \int_{\RN} |\d^\alpha S^L_{N(n)}|^p \,\dd x
- \sum_{l=0}^L \int_{\RN} |\d^\alpha \bbl{w}{l}|^p \,\dd x \r| \\
&\le 
\int_{\RN} \l| \l| \sum_{l=0}^L \d^\alpha \dsl{g}{l}{N(n)} \bbl{w}{l} \r|^p - \sum_{l=0}^L |\d^\alpha \dsl{g}{l}{N(n)} \bbl{w}{l} |^p \r| \,\dd x \\
&\le C_L  \sum_{l \neq l'} 
\int_{\RN} |\d^\alpha \dsl{g}{l}{N(n)} \bbl{w}{l}| |\d^\alpha \dsl{g}{l'}{N(n)} \bbl{w}{l'}|^{p-1} \,\dd x.
\end{align*}
The mutual orthogonality condition~\eqref{eq;500506} implies that for any $0 \le l \neq l' \le L$, 
\begin{equation}
\label{202105020110}
\int_{\RN} |\d^\alpha \dsl{g}{l}{N(n)} \bbl{w}{l}| |\d^\alpha \dsl{g}{l'}{N(n)} \bbl{w}{l'}|^{p-1} \,\dd x
\to 0
\end{equation}
as $n\toinfty$. 
Hence we get 
\begin{equation*}
\lim_{n\toinfty}
\int_{\RN} | \d^\alpha S^L_{N(n)}|^p \,\dd x
= \sum_{l=0}^L \int_{\RN} | \d^\alpha \bbl{w}{l}|^p \,\dd x.
\end{equation*}

Adding up the above over all multi-indices $\alpha$ with $|\alpha| =  m$, we obtain 
\begin{equation}\label{202103150021}
\| S^L_{N(n)} \|_{\Wdmp}^p
= \sum_{l=0}^L \| \bbl{w}{l} \|_{\Wdmp}^p +o(1) \quad (n\toinfty),
\end{equation}
and from~\eqref{eq;20019711} and~\eqref{202103150021}, one sees that 
\begin{align}
\label{202105120010}
\lim_{L\to \Lambda} \lim_{n \toinfty} 
\| S^L_{N(n)} \|_{\Wdmp}^p
&
=\sum_{l=0}^\Lambda \| \bbl{w}{l} \|_{\Wdmp}^p
\\
&\notag 
\le \varlimsup_{n\toinfty} \| u_{N(n)} \|_{\Wdmp}^p.
\end{align}
So we observe  that    
\begin{align*}
&\varlimsup_{L \to \Lambda} \varlimsup_{n\toinfty} 
\|r^L_{N(n)}\|_{\Wdmp} \\ 
&\le 
\varlimsup_{n\toinfty} 
\|u_{N(n)}\|_{\Wdmp}
+
\varlimsup_{L \to \Lambda} \varlimsup_{n\toinfty} 
\|S^L_{N(n)}\|_{\Wdmp} \\ 
&\le 2 \varlimsup_{n\toinfty} \| u_{N(n)} \|_{\Wdmp}, 
\end{align*}
which implies the boundedness of the double sequence 
$(r^L_{N(n)})$ in $\Wdmp$.
\end{proof}
\end{step}

\begin{step}[Vanishing of the residual term]
\begin{lemma}[Vanishing of the residual term~\eqref{residue vanish Wdmp}]
It holds that  
\begin{align*}
&\lim_{L\to \Lambda}\varlimsup_{n\toinfty}\|r^L_{N(n)} \|_{L^{\ppm}(\RN)}=0, \\ 
&\lim_{L\to \Lambda}\varlimsup_{n\toinfty}\|r^L_{N(n)} \|_{\dot{W}^{k,p^*_{m-k}}(\RN)}=0,\quad k\in\NN, \ k<m.
\end{align*}
\end{lemma}

\begin{proof}
This lemma is readily checked by \cref{theorem;weak G-comp conti} and \cref{lemma;G-complete continuity of Wdmp,lemma;G-complete continuity of Wdmp2} together with~\eqref{202104040040}.
\end{proof}
\end{step}

The following two lemmas will provide supplementary calculations for proofs of the above lemmas. 
For the purpose of the proof of~\eqref{202105020090}, 
we generalize the identity as follows:

\begin{lemma}[Generalization of~\eqref{202105020090}]\label{lemma;cut off induction}
For any $L\in \N$ with $L \le \Lambda +1$  and for all 
$l_\nu \in \NN^{<\Lambda +1} \ (1 \le \nu \le L)$  with $l_\nu \neq l_{\nu'} \ (\nu \neq \nu')$,
it holds that 
\begin{equation*}
1 \equiv 
\sum_{\nu=1}^L 
\dsl{\sigma}{l_\nu}{N(n)} \chi_R 
\prod_{l_{\mu} \in J_{l_\nu,l_L}^+} (1- \dsl{\sigma}{l_\mu}{N(n)} \chi_{R} ) 
+ \prod_{\nu=1}^L (1- \dsl{\sigma}{l_\nu}{N(n)} \chi_{R} )
\end{equation*}
for sufficiently large $n \in \NN$,
where 
$\dsl{\sigma}{l_\nu}{N(n)} \chi_{R} $ denotes the characteristic function supported on the ball 
$B(\dsl{y}{l_\nu}{N(n)}, 2^{- \dsl{j}{l_\nu}{N(n)}} R )$, and 
\begin{equation*}
J_{l_\nu,l_L}^+ \coloneqq \qty{ l_\mu \in \NN; \ 1\le \nu \le L, \  \dsl{j}{l_\mu}{N(n)} - \dsl{j}{l_\nu}{N(n)} \to +\infty \ (n\toinfty) }.
\end{equation*}
\end{lemma}

\begin{proof}
We prove the identity by induction on $L$.

\paragraph{(I) Base step: $L=1$}
In this case, the identity is trivial. 

\paragraph{(II) Inductive step}
Let $L \in \N$ satisfy $L +1  \le \Lambda +1$ and $L+1 <+\infty$. 
Suppose that the identity holds true for all positive integers 
up to $L\in \N$. 
Let $l_\nu \in \NN^{<\Lambda +1} \ (1 \le \nu \le L+1)$ with $l_\nu \neq l_{\nu'} \ (\nu \neq \nu')$.
We use the following partition:
\begin{align*}
&\qty{ l_1, \ldots, l_{L} } 
= A_1 \sqcup A_2 \sqcup A_3, \\ 
&A_1 = 
\qty{ l_\nu \in \NN; \ 1 \le \nu \le L, \ 
\dsl{j}{l_\nu}{N(n)} - \dsl{j}{l_{L+1}}{N(n)} \to +\infty \ (n\toinfty) } \\
&\quad = J_{l_{L+1},l_{L+1}}^+, \\
&A_2 = 
\qty{ l_\nu \in \NN; \ 1 \le \nu \le L, \ 
\dsl{j}{l_\nu}{N(n)} - \dsl{j}{l_{L+1}}{N(n)} \to -\infty \ (n\toinfty) }, \\
&A_3 = 
\qty{ l_\nu \in \NN; \ 1 \le \nu \le L, \ 
\dsl{j}{l_\nu}{N(n)} - \dsl{j}{l_{L+1}}{N(n)} \ 
\mbox{is convergent}  }.
\end{align*}
Then we see that 
\begin{alignat*}{2}
&J_{l_\nu,l_{L+1}}^+ = J_{l_\nu,l_L}^+ 
&\quad &\mbox{if} \ l_\nu \in A_1, \\
&J_{l_\nu,l_{L+1}}^+ = J_{l_\nu,l_L}^+ \sqcup \{l_{L+1}\} 
&\quad &\mbox{if} \ l_\nu \in A_2, \\
&J_{l_\nu,l_{L+1}}^+ = J_{l_\nu,l_L}^+ 
&\quad &\mbox{if} \ l_\nu \in A_3.
\end{alignat*}
It follows that 
\begin{align*}
&\sum_{\nu=1}^{L+1} 
\dsl{\sigma}{l_\nu}{N(n)} \chi_R 
\prod_{l_{\mu} \in J_{l_\nu,l_{L+1}}^+} 
(1- \dsl{\sigma}{l_\mu}{N(n)} \chi_{R} ) \\
&=
\sum_{\nu=1}^{L}
\dsl{\sigma}{l_\nu}{N(n)} \chi_R 
\prod_{l_{\mu} \in J_{l_\nu,l_{L+1}}^+} 
(1- \dsl{\sigma}{l_\mu}{N(n)} \chi_{R} )
\\
&\qquad 
+ \dsl{\sigma}{l_{L+1}}{N(n)} \chi_{R} 
\prod_{l_{\mu} \in J^+_{l_{L+1},l_{L+1}} } 
(1- \dsl{\sigma}{l_\mu}{N(n)} \chi_{R} ) \\
&=
\sum_{l_\nu \in A_1} \dsl{\sigma}{l_\nu}{N(n)} \chi_R 
\prod_{l_{\mu} \in J_{l_\nu,l_{L}}^+} 
(1- \dsl{\sigma}{l_\mu}{N(n)} \chi_{R} )
\\
&\qquad 
+  
\sum_{l_\nu \in A_2} \dsl{\sigma}{l_\nu}{N(n)} \chi_R 
\prod_{l_{\mu} \in J_{l_\nu,l_{L}}^+} 
(1- \dsl{\sigma}{l_\mu}{N(n)} \chi_{R} )
(1- \dsl{\sigma}{l_{L+1}}{N(n)} \chi_{R} ) \\ 
&\qq+  
\sum_{l_\nu \in A_3}  \dsl{\sigma}{l_\nu}{N(n)} \chi_R 
\prod_{l_{\mu} \in J_{l_\nu,l_{L}}^+} 
(1- \dsl{\sigma}{l_\mu}{N(n)} \chi_{R} )
\\
&\qquad 
+
\dsl{\sigma}{l_{L+1}}{N(n)} \chi_{R} 
\prod_{l_{\mu} \in A_1 } 
(1- \dsl{\sigma}{l_\mu}{N(n)} \chi_{R} )
 \\
&=
\sum_{\nu = 1}^L 
\dsl{\sigma}{l_\nu}{N(n)} \chi_R 
\prod_{l_{\mu} \in J_{l_\nu,l_{L}}^+} 
(1- \dsl{\sigma}{l_\mu}{N(n)} \chi_{R} )
\\
&\qquad 
- \dsl{\sigma}{l_{L+1}}{N(n)} \chi_{R} 
\sum_{l_\nu \in A_2} \dsl{\sigma}{l_\nu}{N(n)} \chi_{R} \prod_{l_{\mu} \in J_{l_\nu,l_L}^+} (1- \dsl{\sigma}{l_\mu}{N(n)} \chi_{R} ) \\ 
&\qq+  
\dsl{\sigma}{l_{L+1}}{N(n)}
\chi_{R} \prod_{l_{\mu} \in A_1} (1- \dsl{\sigma}{l_\mu}{N(n)} \chi_{R} ). 
\end{align*}

By the induction hypothesis, the last line turns to:  
\begin{align*}
&=
\sum_{\nu = 1}^L 
\dsl{\sigma}{l_\nu}{N(n)} \chi_R 
\prod_{l_{\mu} \in J_{l_\nu,l_{L}}^+} 
(1- \dsl{\sigma}{l_\mu}{N(n)} \chi_{R} )
\\
&\qquad 
- \dsl{\sigma}{l_{L+1}}{N(n)} \chi_{R} 
\sum_{l_\nu \in A_2} \dsl{\sigma}{l_\nu}{N(n)} \chi_{R} \prod_{l_{\mu} \in J_{l_\nu,l_L}^+} (1- \dsl{\sigma}{l_\mu}{N(n)} \chi_{R} ) \\ 
&\qq+  
\dsl{\sigma}{l_{L+1}}{N(n)}  \chi_{R} 
\l(  1 - 
\sum_{l_\nu \in A_1} \dsl{\sigma}{l_\nu}{N(n)} \chi_{R} \prod_{l_{\mu} \in J_{l_\nu,l_L}^+ \cap A_1} (1- \dsl{\sigma}{l_\mu}{N(n)} \chi_{R} )
\r). 
\end{align*}
It is readily seen that 
$J_{l_\nu,l_L}^+ \cap A_1 = J_{l_\nu,l_L}^+$ 
when $l_\nu \in A_1$, and hence, the last line turns to: 
\begin{align} \label{202105220010}
&=  
\sum_{\nu = 1}^L \dsl{\sigma}{l_\nu}{N(n)} \chi_{R}
\prod_{l_{\mu} \in J_{l_\nu,l_L}^+} (1- \dsl{\sigma}{l_\mu}{N(n)} \chi_{R} ) \\ 
&\notag \quad + \dsl{\sigma}{l_{L+1}}{N(n)}  \chi_{R}
\l( 1 - 
\sum_{l_\nu \in A_1 \sqcup A_2 \sqcup A_3} \dsl{\sigma}{l_\nu}{N(n)} \chi_{R} 
\prod_{l_{\mu} \in J_{l_\nu,l_L}^+} (1- \dsl{\sigma}{l_\mu}{N(n)} \chi_{R} ) 
\r) \\
&\notag \qquad + \dsl{\sigma}{l_{L+1}}{N(n)}
\chi_{R}
\sum_{l_\nu \in A_3} \dsl{\sigma}{l_\nu}{N(n)} \chi_{R} \prod_{l_{\mu} \in J_{l_\nu, l_L}^+} (1- \dsl{\sigma}{l_\mu}{N(n)} \chi_{R} ).
\end{align}
Whenever $l_\nu \in A_3$, 
the mutual orthogonality condition~\eqref{eq;500506} implies that the supports of $\dsl{\sigma}{l_{L+1}}{N(n)}\chi_{R}$ and $\dsl{\sigma}{l_\nu}{N(n)} \chi_{R}$ are mutually disjoint for sufficiently large $n$, and thus 
\begin{equation*}
\dsl{\sigma}{l_{L+1}}{N(n)} \chi_{R}
\sum_{l_\nu \in A_3} \dsl{\sigma}{l_\nu}{N(n)}\chi_{R} \prod_{l_{\mu} \in J_{l_\nu,l_L}^+} (1- \dsl{\sigma}{l_\mu}{N(n)} \chi_{R} )
=0
\end{equation*}
for sufficiently large $n$. 

Therefore, from the above and the induction hypothesis again, 
\eqref{202105220010} leads to: 
\begin{align*}
&=  
\sum_{\nu = 1}^L \dsl{\sigma}{l_\nu}{N(n)}\chi_{R} \prod_{l_{\mu} \in J_{l_\nu,l_L}^+} (1- \dsl{\sigma}{l_\mu}{N(n)} \chi_{R} )
\\
&\qquad 
+ \dsl{\sigma}{l_{L+1}}{N(n)} \chi_{R}
\l( 1 - 
\sum_{\nu =1}^L \dsl{\sigma}{l_\nu}{N(n)} \chi_{R} 
\prod_{l_{\mu} \in J_{l_\nu,l_L}^+} (1- \dsl{\sigma}{l_\mu}{N(n)} \chi_{R} ) 
\r) \\
&=
(1 - \dsl{\sigma}{l_{L+1}}{N(n)} \chi_{R} )
\sum_{\nu =1}^L \dsl{\sigma}{l_\nu}{N(n)} \chi_{R}
\prod_{l_{\mu} \in J_{l_\nu,l_L}^+} (1- \dsl{\sigma}{l_\mu}{N(n)} \chi_{R} )
\\
&\qquad 
+ \dsl{\sigma}{l_{L+1}}{N(n)} \chi_{R} \\
&=
(1 - \dsl{\sigma}{l_{L+1}}{N(n)} \chi_{R} )
\l( 1 - \prod_{\nu=1}^L ( 1- \dsl{\sigma}{l_\mu}{N(n)} \chi_{R} )   \r)
+ \dsl{\sigma}{l_{L+1}}{N(n)} \chi_{R} \\
&=
1 - \prod_{\nu=1}^{L+1} ( 1- \dsl{\sigma}{l_\mu}{N(n)} \chi_{R} )
\end{align*}
for sufficiently large $n$.
This completes the proof. 
\end{proof}

\begin{lemma}[Proofs  of~\eqref{202105110010},~\eqref{202105090040} and~\eqref{202105090041}]
\label{lemma;calc in proof of claim E}
Under the same settings as in the proof of \cref{lemma;energy estim Wdmp}, 
it holds that 
\begin{align}
&\label{202105100010}
\prod_{l' \in \Jlp } (1-[\dsl{\sigma}{l}{N(n)}]^{-1}
\dsl{\sigma}{l'}{N(n)} \chi_{R}) \to 1 \s \mbox{in} \ L^q(\RN) 
\ ( q \in [1,+\infty[ ), \\ 
&\label{202105100020}
\int_{\RN} | \d^\alpha \bbl{w}{l} |^p 
 \chi_{R} \prod_{l' \in \Jlp }
 (1-[\dsl{\sigma}{l}{N(n)}]^{-1}
\dsl{\sigma}{l'}{N(n)} \chi_{R})  \,\dd x
 \to \int_{\RN} |\d^\alpha \bbl{w}{l} |^p  \chi_{R} \,\dd x, 
\\
&\label{202105100030}
\int_{\RN} |\d^\alpha \bbl{w}{l} |^{p-2} 
 \d^\alpha \bbl{w}{l} (\d^\alpha ( [ \dsl{g}{l}{N(n)}]^{-1} u_{N(n)} ) - \d^\alpha \bbl{w}{l}  )  
 \\
 &\notag \hspace{3cm}
 \times \chi_{R} \prod_{l' \in \Jlp }
 (1-[\dsl{\sigma}{l}{N(n)}]^{-1}
\dsl{\sigma}{l'}{N(n)} \chi_{R}) \,\dd x
 \to 0,
\end{align}
as $n\toinfty$.
\end{lemma}

\begin{proof}
When $\Jlp = \emptyset$, we always assume that $\prod_{l' \in \Jlp }
(1-[\dsl{\sigma}{l}{N(n)}]^{-1} \dsl{\sigma}{l'}{N(n)} \chi_{R}) \equiv 1$, and so the assertions are all trivial. 
Hence we assume $\Jlp \neq \emptyset$.

We shall prove~\eqref{202105100010}. 
Expanding the products and calculating $1-\prod_{l' \in \Jlp }(1-[\dsl{\sigma}{l}{N(n)}]^{-1} \dsl{\sigma}{l'}{N(n)} \chi_{R})$, 
one sees that each term of the summation contains at least one 
$[\dsl{\sigma}{l}{N(n)}]^{-1} \dsl{\sigma}{l'}{N(n)} \chi_{R}$ for some $l' \in \Jlp$.
Hence the $L^q$-norm of each term is majorized by the $L^q$-norm of such  $[\dsl{\sigma}{l}{N(n)}]^{-1} \dsl{\sigma}{l'}{N(n)} \chi_{R}$ that is infinitesimal as $n\toinfty$, thanks to 
$l' \in \mathcal{J}_{l,L}^+$ and the mutual orthogonality condition. 
Hence~\eqref{202105100010} follows. 

We next show~\eqref{202105100020}. By approximation, we may assume that 
$\d^\alpha\bbl{w}{l}$ is smooth. Then thanks to~\eqref{202105100010}, 
one gets~\eqref{202105100020}.
Regarding~\eqref{202105100030}, we shall approximate 
$|\d^\alpha\bbl{w}{l}|^{p-2}\d^\alpha\bbl{w}{l}$ by a smooth function. 
Then~\eqref{202105100030} follows from 
\begin{align*}
    \d^\alpha [\dsl{g}{l}{N(n)}]^{-1}u_{N(n)}\to\d^\alpha\bbl{w}{l} \s \mbox{weakly in}\ L^p(\RN)
\end{align*}
and~\eqref{202105100010}. 
\end{proof}

Eventually, the proof of \cref{theorem;profile-decomp. in Wdmp} has been complete. 
\qed 

\subsection{Relations among profile decompositions in inhomogeneous and homogeneous Sobolev spaces}\label{section:relations}

Fix $1<p<N/m$, and recall that 
$\Wmp$ is continuously embedded into $\Wdmp$ and that 
$\Wmp\cong\Wdmp\cap L^p(\RN)$.
Now consider that 
\begin{itemize}
\item $(u_n)$ is a bounded sequence in the \emph{inhomogeneous} Sobolev space $\Wmp$.
\end{itemize}
Then there are two types of profile decomposition of $(u_n)$: 
\begin{enumerate}
\item 
the profile decomposition in $\Wmp$ given by~\cite[Theorem~3.8]{Okumura2};

\item 
the profile decomposition in $\Wdmp$ given by \cref{theorem;profile-decomp. in Wdmp}. 
\end{enumerate}
In this section, we shall discuss the relations between the above two. 

\medskip 

Recall that, by~\cref{theorem;profile-decomp. in Wdmp}, there exist:
\begin{itemize}
\setlength{\itemsep}{0mm}
\setlength{\parskip}{0mm}

\item a subsequence of $(n)$, still denoted by $n$, 
\item a number $\Lambda\in\N\cup\qty{0,+\infty}$, 
\item profiles $\bbl{w}{l}\in\Wdmp$, $l\in\NN^{<\Lambda+1}$, 
\item dislocations $(\dsl{y}{l}{n},\dsl{j}{l}{n})\in\RN\times\R$, $l\in\NN^{<\Lambda+1},\ n\in\Z_{\ge l}$, 
\item residual terms $r^L_n\in\Wdmp$, $L\in\NN^{<\Lambda+1}, \ n\in\Z_{\ge L}$, 
\end{itemize}
with the relation of a double-suffix profile decomposition 
\begin{align*}
u_n = \sum_{l=0}^L 2^{\dsl{j}{l}{n}N/\ppm} \bbl{w}{l} \qty( 2^{\dsl{j}{l}{n}}(\cdot-\dsl{y}{l}{n}) ) + r^L_n, \quad 
L\in\NN^{<\Lambda+1}, \ n\in\Z_{\ge L}, 
\end{align*}
such that all relations in~\cref{theorem;profile-decomp. in Wdmp} hold true.
Furthermore, we may assume that 
the index set $\NN^{<\Lambda+1}$ is decomposed into disjoint sets as
\begin{align*}
&\NN^{<\Lambda+1}=J_+\sqcup J_0\sqcup J_-,
\\
&J_+\coloneqq \qty{l\in\NN^{<\Lambda+1};\ \dsl{j}{l}{n}\to+\infty}, \\
&J_0\coloneqq \qty{l\in\NN^{<\Lambda+1};\ \dsl{j}{l}{n}\to\dsl{j}{l}{\infty}}, \\
&J_-\coloneqq \qty{l\in\NN^{<\Lambda+1};\ \dsl{j}{l}{n}\to-\infty}. 
\end{align*}

Under the above situation, we can show the following
\begin{lemma}
\begin{enumerate}
\item 
$J_- = \emptyset$. 

\item 
If $l\in J_0$, then $\bbl{w}{l}\in\Wmp$ and moreover, without loss of generality, one may assume that $\dsl{j}{l}{\infty}=\dsl{j}{l}{n}=0$ for all $n\in\Z_{\ge l}$. 

\end{enumerate}
\end{lemma}

\begin{proof}
We shall first show the assertion~(i). 
Suppose that $J_-\neq\emptyset$ and let $l\in J_-$. It then follows that for every $e\in C^\infty_c(\RN)$,
\begin{align*}
\intrn 2^{-\dsl{j}{l}{n}N/\ppm} u_n \qty( 2^{-\dsl{j}{l}{n}}\cdot+\dsl{y}{l}{n} ) e \,\dd x \to 
\intrn \bbl{w}{l} e \,\dd x \quad (n\toinfty).
\end{align*}
On the other hand, one sees that 
\begin{align*}
\abs{
\intrn 2^{-\dsl{j}{l}{n}N/\ppm} u_n \qty( 2^{-\dsl{j}{l}{n}}\cdot+\dsl{y}{l}{n} ) e \,\dd x
}
&\le 
2^{-\dsl{j}{l}{n}} 
\intrn \qty| 2^{-\dsl{j}{l}{n}N/p} u_n \qty( 2^{-\dsl{j}{l}{n}}\cdot+\dsl{y}{l}{n} ) e | \,\dd x
\\
&\le 
2^{-\dsl{j}{l}{n}} 
\norm{u_n}_{\leb{p}{\RN} }\norm{e}_{\leb{p'}{\RN}} 
\\
&=o(1)
\end{align*}
as $n\toinfty$. 
Thus from the density argument we conclude that $\bbl{w}{l}=0$. 
But this contradicts the fact that $\bbl{w}{l}\neq 0$ for every $l\in\NN^{<\Lambda+1}$, hence the conclusion.

We then prove~(ii). 
Let $l\in J_0$. 
Recall that 
\begin{equation*}
    2^{-\dsl{j}{l}{n}N/\ppm}u_n(2^{-\dsl{j}{l}{n}}\cdot+\dsl{y}{l}{n})
    \to\bbl{w}{l}\ \mbox{weakly in}\ \Wdmp. 
\end{equation*}
Since $\dsl{j}{l}{n}\to\dsl{j}{l}{\infty}$, one sees from changing variables that 
\begin{align*}
2^{-(\dsl{j}{l}{n}-\dsl{j}{l}{\infty})N/\ppm}u_n(2^{-(\dsl{j}{l}{n}-\dsl{j}{l}{\infty})}\cdot+\dsl{y}{l}{n})
\to 2^{\dsl{j}{l}{\infty}M/\ppm}\bbl{w}{l} \qty(2^{\dsl{j}{l}{\infty}} \cdot )
\\
\mbox{weakly in}\ \Wdmp. 
\end{align*}
Thus with no loss of generality, we assume that 
$\dsl{j}{l}{\infty}=0$ for each $l\in J_0$ and $\dsl{j}{l}{n}\to 0$.

Take $\phi\in \qty[\Wdmp]^*$ arbitrarily, and let $(\phi_\alpha)_{\alpha}\subset\leb{p}{\RN}$ be its Riesz representation given by~\cref{lemma;dual Wdmp}. 
It then follows that 
\begin{align*}\eqntag\label{202201240010}
\sum_{|\alpha|=m}
\intrn \partial_\alpha \qty( 2^{-\dsl{j}{l}{n}N/\ppm} u_n \qty( 2^{-\dsl{j}{l}{n}} \cdot +\dsl{y}{l}{n} ) ) \phi_\alpha \,\dd x
\to 
\sum_{|\alpha|=m}
\intrn \partial_\alpha \bbl{w}{l} \phi_\alpha \,\dd x
\end{align*}
as $n\toinfty$. 
On the other hand, it follows that 
\begin{align*}\eqntag\label{202201240020}
&\abs{
\intrn \partial_\alpha \qty( 2^{-\dsl{j}{l}{n}N/\ppm} u_n \qty( 2^{-\dsl{j}{l}{n}}\cdot+\dsl{y}{l}{n} ) ) \phi_\alpha \,\dd x 
-
\intrn \partial_\alpha  u_n \qty( \cdot+\dsl{y}{l}{n} )  \phi_\alpha \,\dd x 
}
\\
&\le 
\intrn 
\abs{
\partial_\alpha u_n (\cdot+\dsl{y}{l}{n})
}
\abs{
2^{\dsl{j}{l}{n}N/\ppm}\phi_\alpha (2^{\dsl{j}{l}{n}}\cdot)\phi_\alpha
-\phi_\alpha
}\,\dd x
\\
&\le 
\norm{u_n}_{\Wdmp}
\norm{
2^{\dsl{j}{l}{n}N/\ppm}\phi_\alpha (2^{\dsl{j}{l}{n}}\cdot)\phi_\alpha
-\phi_\alpha
}_{\leb{p'}{\RN}}
\\
&=o(1)
\end{align*}
as $n\toinfty$. 

From~\eqref{202201240010} and~\eqref{202201240020}, we infer that 
\begin{align*}
\sum_{|\alpha|=m}
\intrn \partial_\alpha u_n \qty( 2^{-\dsl{j}{l}{n}} \cdot +\dsl{y}{l}{n}) \phi_\alpha \,\dd x
\to 
\sum_{|\alpha|=m}
\intrn \partial_\alpha \bbl{w}{l} \phi_\alpha \,\dd x
\end{align*}
as $n\toinfty$, which implies that 
\begin{align*}
u_n (\cdot+\dsl{y}{l}{n}) \to \bbl{w}{l} \s \text{weakly in}\ \Wdmp.
\end{align*}

Also, from the following lemma, one finds that 
 $\bbl{w}{l}\in\Wmp$ (not just $\Wdmp$) for all $l\in J_0$, 
 and thus the proof is complete. 
\end{proof}

\begin{lemma}
Let $u_n\in\Wmp$ $(n\in\N)$ be such that 
$u_n\to u$ weakly in $\Wdmp$ (not $\Wmp$)
for some $u\in\Wdmp$. 
If $\sup_{n\in\N}\|u_n\|_{\Wmp}<\infty$, then 
$u\in\Wmp.$
\end{lemma}

\begin{proof}
It suffices to show that $u_n\to u$ weakly in $L^p(\RN)$. 
This follows from the boundedness of $(u_n)$ in $L^p(\RN)$ and the pointwise convergence $u_n\to u$ a.e. in $\RN$ (up to a subsequence) which follows from the local compact embeddings for Sobolev spaces. 
\end{proof}

By the same calculations as in the proof of the energy decomposition for the inhomogeneous Sobolev norm (see also~\cite{Okumura2}), one can show that
\begin{equation*}
    \varlimsup_{n\toinfty}\|u_n\|_{\Wmp}^p
    \ge \sum_{l\in J_0}\|\bbl{w}{l}\|_{\Wmp}^p.
\end{equation*}
Also, we find that 
\begin{equation*}
    \lim_{L\to\Lambda}\sup_{\phi\in B_{[\Wmp]^*}(1)}\varlimsup_{n\toinfty}\sup_{y\in\RN}
    \abs{\la \phi, \rho^L_n(\cdot+y) \ra_{\Wmp}}
    =0,
\end{equation*}
where $\rho^L_n=u_n-\sum_{l=0}^L\bbl{w}{i_l}(\cdot-\dsl{y}{i_l}{n})$ and 
$J_0=\{0<i_1<i_2<\cdots\}$.
This can be proved by the same argument as in~\cite[Lemma~3.12]{Okumura2}.

Therefore, $\bbl{w}{l}$ and $\dsl{y}{l}{n}$ ($l\in J_0$, $n\in\N_{\ge l}$) are profile elements for the profile decomposition of $(u_n)$ in $\Wmp$,
and assertions as in~\cite[Theorem~3.8]{Okumura2} hold.
In short, the profile decompositions in $\Wdmp$ of bounded sequences in $\Wmp$ include the profile decompositions in $\Wmp$ of $(u_n)$. 

\medskip 

Summing up the above arguments, we obtain the following 
\begin{theorem}
Let $(u_n)$ be a bounded sequence in $\Wmp$. Then there exist
 a subsequence of $(n)$, still denoted by $n$, 
 a number $\Lambda\in\N\cup\qty{0,+\infty}$, 
 profiles $\bbl{w}{l}\in\Wdmp$ $(l\in\NN^{<\Lambda+1})$, 
 dislocations $(\dsl{y}{l}{n},\dsl{j}{l}{n})\in\RN\times\R$ $(l\in\NN^{<\Lambda+1}, \ n\in\Z_{\ge l})$, 
 and residual terms $r^L_n\in\Wdmp$ $(L\in\NN^{<\Lambda+1}, \ n\in\Z_{\ge L})$, 
with the relation of a double-suffix profile decomposition 
\begin{align*}
u_n = \sum_{l=0}^L 2^{\dsl{j}{l}{n}N/\ppm} \bbl{w}{l} \qty(2^{\dsl{k}{l}{n}} (\cdot-\dsl{y}{l}{n}) ), \quad 
L\in\NN^{<\Lambda+1},\ n\in\Z_{\ge L}, 
\end{align*}
such that assertions as in~\cref{theorem;profile-decomp. in Wdmp} hold true. 
Moreover, set 
\begin{align*}
   &J_+\coloneqq \{l\in\NN^{<\Lambda+1};\ \dsl{j}{l}{n}\to+\infty\}, \\
   &J_0\coloneqq \{l\in\NN^{<\Lambda+1};\ \dsl{j}{l}{n}\to\dsl{j}{l}{\infty}\}, \\
   &J_-\coloneqq \{l\in\NN^{<\Lambda+1};\ \dsl{j}{l}{n}\to-\infty\}, \\
   &\NN^{<\Lambda+1}=J_+\sqcup J_0\sqcup J_-.
\end{align*}
Then, 
\begin{align*}
&J_-=\emptyset, \\ 
&\begin{alignedat}{2}
    &\dsl{j}{l}{n}=\dsl{j}{l}{\infty}=0 &\quad &(l\in J_0), \\ 
    &\bbl{w}{l}\in\Wmp &\quad &(l\in J_0), \\
    &u_n(\cdot+\dsl{y}{l}{n})\to\bbl{w}{l} &\quad &\mbox{weakly in}\ \Wmp\ (l\in J_0),\\
\end{alignedat}\\
    &\varlimsup_{n\toinfty}\|u_n\|_{\Wmp}^p
    \ge \sum_{l\in J_0}\|\bbl{w}{l}\|_{\Wmp}^p, \\
    &\lim_{L\to\Lambda}\sup_{\phi\in B_{[\Wmp]^*}(1)}\varlimsup_{n\toinfty}\sup_{y\in\RN}
    \abs{\la \phi, \rho^L_n(\cdot+y)  \ra _{\Wmp}}
    =0, \\ 
    &\lim_{L\to\Lambda}\varlimsup_{n\toinfty}\|\rho^L_n\|_{L^q(\RN)}=0,\quad q\in]p,\ppm[,
\end{align*}
where $\rho^L_n\coloneqq u_n-\sum_{l=0}^L\bbl{w}{i_l}(\cdot-\dsl{y}{i_l}{n})$ and 
$J_0=\{0<i_1<i_2<\cdots\}$.
\end{theorem}



\paragraph{Summary}
The above arguments are summarized as follows:
For a bounded sequence in $\Wdmp$ has three types of profiles:
\begin{itemize}
\setlength{\itemsep}{0mm}
\setlength{\parskip}{0mm}

\item ones which are translated by vectors $\dsl{y}{l}{n}\in\RN$,
\item ones which are concentrating at some points, 
\item ones which are anti-concentrating, i.e., vanishing locally. 
\end{itemize}
However, the above arguments indicate that 
if $(u_n)$ is bounded in $\leb{p}{\RN}$, then the anti-concentrating profiles do not exist because the $L^p$-norm of those profiles is increasing to infinity, contradicting the $L^p$-boundedness. 
Conversely, the concentrating profiles exist because the $L^p$-norm of those profiles is converging to zero.

\section[Decomposition of integral functionals]{Decomposition of integral functionals in critical cases}\label{BLcrit}
In this section, we always assume $1 <  p < N/m $ for $m,N \in \N$.

\medskip
As is intended in~\cite{Okumura2}, we shall investigate the results of decomposition of integral functionals of critical order (in other words, the iterated Brezis-Lieb lemma). 
The author pointed out in~\cite{Okumura2} that the results of decomposition of integral functionals will be obtained in Lebesgue or Sobolev spaces into which the dislocation Sobolev spaces considered are embedded $G$-completely continuously. Hence, along with \cref{lemma;G-complete continuity of Wdmp,lemma;G-complete continuity of Wdmp2}, we shall approach the results of decomposition of integral functionals in lower order homogeneous Sobolev spaces with critical exponents.
This result has been obtained by some precursors on profile decomposition, e.g.,~\cite{T-01,T-K}. 

\subsection{Brief Summary}
Firstly, we shall briefly review our results, choosing typical and simple examples in order to describe the essence of the results. 
Assume that $(u_n)$ is a bounded sequence in $\Wdmp$ and take  profile elements $(\bbl{w}{l}, \dsl{y}{l}{n}, \dsl{j}{l}{n} \Lambda)$ (on a subsequence still denoted by $n$) given by \cref{theorem;profile-decomp. in Wdmp}.
According to \cref{lemma;G-complete continuity of Wdmp,lemma;G-complete continuity of Wdmp2} and the exactness condition~\eqref{eq;Ishiwata condition Wdmp}, the residual term $r^L_n$ and its derivatives are vanishing (as $n\toinfty$ and then $L\to\Lambda$) in suitable $L^{\ppm}(\RN)$ or $\dot{W}^{k,p^*_{m-k}}(\RN)$. 

Typically, our main results read: 
\begin{alignat}{2}
&\lim_{n\to\infty}\int_{\RN} |u_n|^{\ppm} \,\dd x
=\sum_{l=0}^\Lambda \int_{\RN} |\bbl{w}{l}|^{\ppm} \,\dd x, & & \label{massdeco11}\\ 
&\lim_{n\toinfty}\int_{\RN}|\d^\alpha u_n|^{p^*_{m-|\alpha|}} \,\dd x
=\sum_{l=0}^\Lambda \int_{\RN} |\d^\alpha \bbl{w}{l}|^{p^*_{m-|\alpha|}} \,\dd x, 
&\quad  &|\alpha|<m, \label{massdeco12}\\ 
&\lim_{n\toinfty}\|u_n\|_{W^{k,p^*_{m-k}}(\RN)}^{p^*_{m-k}}
=\sum_{l=0}^\Lambda \|\bbl{w}{l}\|_{W^{k,p^*_{m-k}}(\RN)}^{p^*_{m-k}}, &\quad &0\le k<m. \label{massdeco13}
\end{alignat}

These formulas are shown as follows.  
As for Euclidean norms, one can show: 
\begin{alignat}{2}
&\l| \l|\sum_{l=1}^L s_l\r|^q-\sum_{l=1}^L|s_l|^q\r|
\le C_L \sum_{l\neq k}|s_l||s_k|^{q-1}, &\quad &s_l\in\R, \label{massdecoineq11}\\ 
&| |a+b|^q-|a|^q|\le \eps |a|^q+C_\eps|b|^q, &\quad &\eps>0, \ a,b\in\R. \label{massdecoineq12}
\end{alignat}
Combining the above, we see: 
\begin{align*}
&\l|\int_{\RN}|u_n|^{\ppm} \,\dd x 
-\sum_{l=0}^L\int_{\RN}|\bbl{w}{l}|^{\ppm} \,\dd x \r| \\
&=\l|\int_{\RN}|u_n|^{\ppm} \,\dd x 
-\sum_{l=0}^L\int_{\RN}|\bbl{w}{l}(\cdot-\dsl{y}{l}{n})|^{\ppm} \,\dd x \r| \\
&\le \eps\int_{\RN}\l|\sum_{l=0}^L\bbl{w}{l}(\cdot-\dsl{y}{l}{n})\r|^{\ppm} \,\dd x  
+C_\eps\int_{\RN}|r^L_n|^{\ppm} \,\dd x 
\\
&\qquad 
+C_L\sum_{l\neq k}\int_{\RN}|\bbl{w}{l}(\cdot-\dsl{y}{l}{n})||\bbl{w}{k}(\cdot-\dsl{y}{k}{n})|^{\ppm-1} \,\dd x .
\end{align*}

Since $(u_n)$ and $(r^L_n)$ are bounded in $\Wdmp$, 
$\int_{\RN}|\sum_{l=0}^L\bbl{w}{l}(\cdot-\dsl{y}{l}{n})|^{\ppm} \,\dd x $ is bounded. 
By the mutual orthogonality condition, $\int_{\RN}|\bbl{w}{l}(\cdot-\dsl{y}{l}{n})||\bbl{w}{k}(\cdot-\dsl{y}{k}{n})|^{\ppm-1} \,\dd x $ is converging to zero as $n\toinfty$. 
Also, the residual term $\int_{\RN}|r^L_n|^{\ppm} \,\dd x $ is vanishing as $n\toinfty$ and then $L\to\Lambda$. 
Therefore, passing to the limits as $n\toinfty$, $L\to\Lambda$ and then $\eps\to 0$, one can conclude~\eqref{massdeco11}.
Similarly, one can show~\eqref{massdeco12} and~\eqref{massdeco13}.

In the above observation, the inequality~\eqref{massdecoineq11} is employed together with the mutual orthogonality condition, yielding the degeneracy of cross terms. The inequality~\eqref{massdecoineq12} is used with $b=r^L_n$ which is vanishing in the suitable space. 
Hence, one can obtain other variants of the above decompositions of integral functionals, by considering the Lebesgue or Sobolev spaces where $\Wdmp$ is embedded $G[\RN,\Z;\ppm]$-completely continuously. 
Along the above strategy, in what follows, we shall investigate more general results.

\subsection{Main theorems}
Now we shall discuss the general results of decomposition of integral functionals of $(\intrn F(\d^\alpha u_n) \,\dd x )$ with continuous functions $F$ of critical growth, based on the profile decomposition in \emph{homogeneous} Sobolev spaces.  
We consider two types of continuous functions here: 
locally Lipschitz continuous functions with homogeneity of critical order; smooth functions with asymptotic growth of critical power near zero and infinity.
We refer the reader to~\cite[Section~5.2.]{T-K} as a reference of conditions of continuous functions.

\smallskip
In what follows, ${\rm Lip}_{\rm loc}(\R)$  denotes the set of locally Lipschitz continuous functions on $\R$. 
Let $\alpha\in(\NN)^N$ be a multi-index such that $|\alpha|<m$. 
Firstly, consider a function $F_\alpha\in {\rm Lip}_{\rm loc}(\R)$ such that 
\begin{equation}\label{growth cond Wdmp 1}
F_\alpha(2^{jN/p^*_{m-|\alpha|}}s)=2^{jN}F_\alpha(s), \quad s\in\R, \ j\in\Z.
\end{equation}
From this condition immediately follows 
\begin{equation*}
|F(s)|\le C |s|^{p^*_{m-|\alpha|}}, \quad s\in \R,
\end{equation*}
for some constant $C>0$. 
For those functions $F$, 
an invariance property follows from the change of  variables:  
\begin{equation*}
\int_{\RN} F \qty( \d^\alpha \qty[2^{jN/\ppm} u(2^j x) ]) \,\dd x 
=
\int_{\RN} F( \d^\alpha u(x) ) \,\dd x 
\end{equation*}
for all $u\in L^{\ppm}(\RN)$ and $j\in\Z$.

Now one form of the results of decomposition of integral functionals in the critical case reads:

\begin{theorem}[Decomposition of integral functionals~1]\label{theorem;B-L Wdmp 1}
Let $(u_n)$ be a bounded sequence in $\Wdmp$.  
Assume that on a subsequence, still denoted by $n$, $(u_n)$ has a profile decomposition with profile elements  
$(\bbl{w}{l}, \dsl{y}{l}{n}, \dsl{j}{l}{n}, \Lambda) \in  \Wdmp\times   \RN\times\Z\times (\N\cup\{0,\infty\})$  
$( l \in \NN^{<\Lambda +1}$, $n \in\Z_{\ge l})$  
as in \cref{theorem;profile-decomp. in Wdmp}.
Let $\alpha\in(\NN)^N$ be a multi-index such that $|\alpha|<m$, and 
let $F_\alpha\in {\rm Lip}_{\rm loc}(\R)$ satisfy~\eqref{growth cond Wdmp 1}. 
Then the following holds true: 
\begin{equation}\notag 
\lim_{n\toinfty}\int_{\RN}F_\alpha(\d^\alpha u_n(x)) \,\dd x 
=
\sum_{l=0}^\Lambda \int_{\RN} F_\alpha(\d^\alpha \bbl{w}{l}(x)) \,\dd x .
\end{equation}
\end{theorem}

We secondly consider a continuous function $F_\alpha$ as follows. 
Let $f_\alpha\in C(\R)$ satisfy  
\begin{equation}\label{10240010}
    |f_\alpha(s)| \le C|s|^{(p^*_{m-|\alpha|})-1}, \quad s\in \R,
\end{equation}
for some $C>0$.
Let $F_\alpha\in C^{1}(\R)$ be the primitive function of $f_\alpha$ given by   
$F_\alpha(s)=\int_0^s f_\alpha(t) \, \dd t, \s s\in\R$. 
From~\eqref{10240010}, it follows that 
\begin{equation*}
|F_\alpha(s)| \le C |s|^{p^*_{m-|\alpha|}}, \quad s \in \R.
\end{equation*}
Now assume that the following limits exist: 
\begin{equation}\label{limiting state modulus B-L 4}
\begin{alignedat}{2}
a^{+} &\coloneqq \lim_{s\to 0^+} |s|^{-p^*_{m-|\alpha|}} F_\alpha(s),  &\qq
A^{+} &\coloneqq \lim_{s\toinfty} |s|^{-p^*_{m-|\alpha|}} F_\alpha(s), \\
a^{-} &\coloneqq \lim_{s\to 0^-} |s|^{-p^*_{m-|\alpha|}} F_\alpha(s), &\qquad 
A^{-} &\coloneqq \lim_{s\to -\infty} |s|^{-p^*_{m-|\alpha|}} F_\alpha(s).
\end{alignedat}
\end{equation}
Set limit functions of $F$ by 
\begin{align}
\label{limiting state B-L 41}
F_{\alpha,0}(s) &\coloneqq 
\begin{cases}
    a^{+}|s|^{p^*_{m-|\alpha|}}  &\mbox{if\ } s\ge 0, \\
    a^{-}|s|^{p^*_{m-|\alpha|}} &\mbox{if \ } s<0,
\end{cases} \\
\label{limiting state B-L 42}
F_{\alpha,\infty} (s) &\coloneqq 
\begin{cases}
    A^{+}|s|^{p^*_{m-|\alpha|}}  &\mbox{if\ } s\ge 0, \\
    A^{-}|s|^{p^*_{m-|\alpha|}} &\mbox{if \ } s<0.
\end{cases}
\end{align}

With these functions, a decomposition of integral functionals also holds true, that is,

\begin{theorem}[Decomposition of integral functionals~2] \label{theorem;B-L 4}
Let $(u_n)$ be a bounded sequence in $\Wdmp$. 
Assume that on a subsequence, still denoted by $n$, $(u_n)$ has a  profile decomposition with profile elements  $(\bbl{w}{l}, \dsl{y}{l}{n}, \dsl{j}{l}{n},\Lambda) \in \Wdmp \times \RN \times \Z\times (\N\cup\{0,\infty\})$  $( l \in \NN^{<\Lambda +1}$, $n \in\Z_{\ge l})$ as in \cref{theorem;profile-decomp. in Wdmp}.
Let $\alpha\in(\NN)^N$ be a multi-index such that $|\alpha|<m$, and let $f_\alpha \in C(\R)$ satisfy~\eqref{10240010} and let 
$F_\alpha(s)\coloneqq \int_0^s f_\alpha(t) \, \dd t$, $s \in \R$,  
admit the limits~\eqref{limiting state modulus B-L 4}. 
Define $F_{\alpha,0}$ and $F_{\alpha,\infty}$ by~\eqref{limiting state B-L 41} and~\eqref{limiting state B-L 42}.
Then, on a subsequence again if necessary, still denoted by $n$,  
the following holds true: 
\begin{align}
&\label{eq;B-L 4} 
\lim_{n\toinfty} \int_{\RN} F_\alpha(\d^\alpha u_n(x)) \,\dd x \\
&\notag 
\quad = \sum_{l \in \N_1} \int_{\RN} F_\alpha \qty( \d^\alpha \qty[2^{ \bbl{j}{l}N/\ppm } \bbl{w}{l}(2^{\bbl{j}{l}} x)]  )  \,\dd x \\
&\notag 
\qquad +
\sum_{l \in \N_2} \int_{\RN} F_{\alpha,\infty}(\d^\alpha \bbl{w}{l}(x)) \,\dd x
+
\sum_{l \in \N_3} \int_{\RN} F_{\alpha,0}(\d^\alpha \bbl{w}{l}(x)) \,\dd x,
\end{align}
where $\N_1 \cup \N_2 \cup \N_3 = \NN^{<\Lambda +1}$,
$\dsl{j}{l}{n} \to \bbl{j}{l}$ if $l \in \N_1$,  
$\dsl{j}{l}{n} \toinfty$ if $l \in \N_2$, 
$\dsl{j}{l}{n} \to -\infty$ if $l\in \N_3$. 
\end{theorem}

One can prove other variants of the above results in the same way as follows:

\begin{proposition}\label{prop;massdecomp11}
Along the profile decomposition in $\Wdmp$ $(1<m<N/p)$, there holds 
\begin{gather*}
\lim_{n\toinfty}\|u_n\|_{\dot{W}^{k,q}(\RN)}^q=\sum_{l=0}^\Lambda \|\bbl{w}{l}\|_{\dot{W}^{k,q}(\RN)}^q, \quad q=p^*_{m-k}, \  k\in\NN, \ k<m.
\end{gather*}
\end{proposition}

\subsection{Proof of \cref{theorem;B-L Wdmp 1}}
We shall prove the above results in the special case $|\alpha|=0$ and we abbreviate $f_\alpha$ to $f$ and so on; the other cases will be readily shown in the same way.

\paragraph{Basic lemmas}
We need important inequalities which play  crucial roles, corresponding to~\eqref{massdecoineq11} and~\eqref{massdecoineq12}.

\begin{lemma}\label{lemma;10180009}
Let $F\in {\rm Lip}_{\rm loc}(\R)$ be as in \cref{theorem;B-L Wdmp 1}.
Then for any $L\in\N$, there exists a constant $C_L >0$ such that for all $s_l \in \R, \ 1\le l \le L$,  
\begin{equation}\label{10180009}
\l| F \l( \sum_{l=1}^L s_l \r) -\sum_{l=1}^L F(s_l) \r|
\le C_L \sum_{1\le l\neq k \le L} |s_l||s_k|^{\ppm -1}.
\end{equation}
\end{lemma}

\begin{proof}
We argue by  induction on $L$.
For the sake of convenience, we set 
$d=2^{N/\ppm}>1$, 
$M\coloneqq [-d^2,-d^{-1}]\cup[d^{-1},d^2] \subset \R$ and  
$ L_M \coloneqq {\rm Lip}(M)$ which is the Lipschitz constant of $F$ on $M$. Fix $\alpha >0$ sufficiently large so that 
\begin{equation*}
1+d^{-\alpha} \le d, \quad 1-d^{-\alpha} \ge d^{-1}.
\end{equation*}
Note that this is equivalent to  
\begin{equation*}
\alpha > \max \l\{ 0, -\frac{\log (d-1)}{\log d}, -\frac{\log (1-1/d)}{\log d}  \r\}.
\end{equation*}

\paragraph{(I) Base step: $L=2$}
We shall show that there exists $C>0$ such that for any $a,b \in \R$,  
\begin{equation}\label{10180010}
|F(a+b)-F(a)-F(b)| \le C(|a|^{\ppm-1}|b|+|a||b|^{\ppm-1}).
\end{equation}
If $a=0$ or $b=0$, then~\eqref{10180010} is immediate, so we assume $a, b\neq 0$.
Firstly, we consider the case where the ratio $t= |b|/|a|$ is 
so small that $t\le d^{-\alpha}$.
Also we suppose $|a|\in [1,d]$. Then one has 
$a+b, a \in M$ since 
\begin{align*}
&|a+b| \le |a| + |b| \le d(1+d^{-\alpha}) \le d^2, \\
&|a+b| \ge |a| - |b| \ge 1(1-d^{-\alpha}) \ge d^{-1}.
\end{align*}
It follows that 
\begin{align}
\label{10180020}
&|F(a+b)-F(a)-F(b)|  \\
&\notag \le 
|F(a+b)-F(a)|+|F(b)|  
\le 
L_M|b|+C|b|^{\ppm} \\
&\notag \le 
L_M|a|^{\ppm-1}|b| +C |a|^{\ppm-1}|b| 
\le (L_M+C) |a|^{\ppm-1}|b|,
\end{align}
where we used the assumptions 
$|a|\in [1,d]$ and $|b|/|a| \le d^{-\alpha}$.

Then we remove the restriction $|a|\in [1,d]$,  and 
fix arbitrary $a\in\R$. Then there is a unique $j = j_a \in\Z$ such that 
$|a|\in [d^j,d^{j+1}]$, and so one sees 
$|d^{-j}a|\in[1,d]$.
Due to $t\le d^{-\alpha}$,  
it follows that $d^{-j}a+d^{-j}b, d^{-j}a \in M$
since 
\begin{align*}
&|d^{-j}a+d^{-j}b| \le d^{-j}|a|+d^{-j}|b| \le d^{-j}|a|(1+d^{-\alpha}) \le d (1+d^{-\alpha}) \le d^2, \\
&|d^{-j}a+d^{-j}b| \ge d^{-j}|a| - d^{-j}|b| \ge d^{-j}|a|(1-d^{-\alpha}) \ge 1(1-d^{-\alpha}) \le d^{-1}.
\end{align*}
Hence from~\eqref{growth cond Wdmp 1} and~\eqref{10180020} one sees that 
\begin{align*}
&|F(a+b)-F(a)-F(b)|  \\
&\le 
|F(d^{j}(d^{-j}a+d^{-j}b))-F(d^j(d^{-j}a))|+|F(d^j(d^{-j}b))|  \\
&=
2^{jN}|F(d^{-j}a+d^{-j}b)-F(d^{-j}a)|+ 2^{jN}|F(d^{-j}b)|  \\
&\le 
2^{jN} (L_M+C) |d^{-j}a|^{\ppm-1}|d^{-j}b| \\
&=
(L_M+C) |a|^{\ppm-1}|b|.
\end{align*}
Thus~\eqref{10180010} holds true if $a\in\R$ and $t\le d^{-\alpha}$.
In the same way, one can show that
\begin{equation*}
|F(a+b)-F(a)-F(b)|  
\le 
(L_M+C) |a||b|^{\ppm-1}
\end{equation*}
if $b\in\R$ and $t \ge d^{\alpha}$.

We finally have to verify~\eqref{10180010} in the case of 
$a,b \in \R $ with $ d^{-\alpha} \le t \le d^{\alpha}$.
It follows that 
\begin{align*}
|F(a+b)-F(a)-F(b)|
&\le 
|F(a+b)|+|F(a)|+|F(b)| \\
&\le 
C_p (|a|^{\ppm}+|b|^{\ppm}) \\
&=
C_p |a|^{\ppm} (1+t^{\ppm}).
\end{align*}
There exists a constant $C>0$ only depending on $d^{\alpha}$ such that 
\begin{equation*}
    1+t^{\ppm} \le C( t+t^{\ppm-1} ), \quad t\in [d^{-\alpha},d^{\alpha}],
\end{equation*}
and thus, we get  
\begin{align*}
    |F(a+b)-F(a)-F(b)|
    &\le 
    C |a|^{\ppm}( t+t^{\ppm-1} ) \\
    &=
    C(|a|^{\ppm-1}|b|+|a||b|^{\ppm-1})
\end{align*}
for all $a,b\in \R$ such that $t\in [d^{-\alpha},d^{\alpha}]$.
Eventually, \eqref{10180010} is verified for all $a,b \in \R$.

\paragraph{(II) Inductive step}
Assume that \eqref{10180009} holds true for some $L\in\N$. 
Set $a=\sum_{l=1}^L s_l$.
Then by the induction hypothesis and~\eqref{10180010}, 
one sees that 
\begin{align*}
    &\l| F \l( \sum_{l=1}^{L+1} s_l \r) -\sum_{l=1}^{L+1} F(s_l) \r| \\
    &\le 
    \l| F \l( a+s_{L+1} \r) -F(a) - F(s_{L+1}) \r| 
    + \l| F \l( \sum_{l=1}^L s_l \r) -\sum_{l=1}^L F(s_l) \r| \\
    &\le 
    C ( |a|^{\ppm-1}|s_{L+1}|+|a||s_{L+1}|^{\ppm-1}  )
    +C_L \sum_{1\le l\neq k \le L}| s_l| |s_k|^{\ppm-1} \\
    &\le 
    C_{L+1} \sum_{1\le l\neq k \le L+1} |s_l||s_k|^{\ppm -1}.
\end{align*}
Hence \eqref{10180009} with $L$ being replaced with $L+1$ is verified.
Hence~\eqref{10180009} is proved for the case $L\ge 2$. 
The case $L=1$ is obviously verified. Thus the proof is complete. 
\end{proof}

\begin{lemma}\label{lemma;10180030}
Let $F\in {\rm Lip}_{\rm loc}(\R)$ be as in \cref{theorem;B-L Wdmp 1}.
Then for any $\eps>0$, there exists a constant $C=C_\eps >0$ such that for all $a,b \in \R$, 
\begin{equation}\notag 
\l| F \l(a+b \r) -F(a) \r|
\le \eps |a|^{\ppm } +C|b|^{\ppm}.
\end{equation}
\end{lemma}

\begin{proof}
Fix $\eps>0$ arbitrarily.
The preceding lemma and the Young inequality yield  
\begin{align*}
    \l| F \l(a+b \r) -F(a) \r|
    &\le 
    \l| F \l(a+b \r) -F(a)-F(b) \r| + |F(b)| \\
    &\le 
    C(|a|^{\ppm-1}|b|+|a||b|^{\ppm-1}) + C|b|^{\ppm} \\
    &\le 
    \eps |a|^{\ppm } +C_{\eps} |b|^{\ppm},
\end{align*}
hence the conclusion. 
\end{proof}

Combining the above two lemmas, one gets the following 

\begin{lemma}
Let $F \in {\rm Lip}_{\rm loc}(\R)$ be as in \cref{theorem;B-L Wdmp 1} and let $L \in  \N$.
Then for any $\eps >0$, there exists $C_\eps, C_L >0$ such that for all $s_l, r \in \R$ $(l=1,\ldots, L)$, 
\begin{align}
\label{202105080050}
&\l| F \l( \sum_{l=1}^L s_l +r  \r) - \sum_{l=1}^L F(s_l) \r|  
\\
&\notag 
\le 
\eps \l|\sum_{l=1}^L s_l \r|^{\ppm} +C_\eps |r|^{\ppm} 
+C_L \sum_{1 \le l \neq l' \le L} |s_l||s_{l'} |^{\ppm-1}. 
\end{align}
\end{lemma}

\paragraph{Main body}
Now we are ready to prove the first result of decomposition of integral functionals.

\begin{proof}[Proof of \cref{theorem;B-L Wdmp 1}]
Set 
$$
E \coloneqq \sup_{n \ge 0}\|u_n\|_{\Wdmp}, \quad 
S^L_n \coloneqq \sum_{l=0}^L 
\dsl{g}{l}{n} \bbl{w}{l}, 
\quad  L \in \NN^{<\Lambda+1}, \  n \in \Z_{\ge L}, 
$$
and one has 
$
u_n = S^L_n + r^L_n.
$
Fix $\eps>0$ arbitrarily.
From~\eqref{202105080050}, one sees that 
\begin{align}
&\label{202105080060}
\int_{\RN} \l|
F(u_n) - \sum_{l=0}^L F(\dsl{g}{l}{n} \bbl{w}{l} )
\r| \,\dd x 
\\
&\le 
\eps \| S^L_n \|^{\ppm}_{L^{\ppm}(\RN)}
+C_\eps \| r^L_n \|^{\ppm}_{L^{\ppm}(\RN)} \notag  
\\
&\notag \qquad 
+C_L\sum_{1 \le l \neq l' \le L} 
\int_{\RN}
|\dsl{g}{l}{n} \bbl{w}{l}  | 
|\dsl{g}{l'}{n} \bbl{w}{l'}|^{\ppm-1} \,\dd x.  \notag 
\end{align}
By~\eqref{eq;residue bounded Wdmp},~\eqref{202105120010} and the Sobolev inequality, we have 
\begin{equation}
\varlimsup_{L\to \Lambda} \varlimsup_{n\toinfty} 
\|S^L_n\|_{L^{\ppm}(\RN)} 
\le C
 \varlimsup_{L\to \Lambda} \varlimsup_{n\toinfty} \|S^L_n\|_{\Wdmp} 
\le C E. \label{202103170110}  
\end{equation}
The mutual orthogonality condition (the assertion~(ii) in \cref{theorem;profile-decomp. in Wdmp}) 
implies that for all $l \neq l' \in \NN^{<\Lambda+1}$,  
\begin{equation}\label{202105080070}
\int_{\RN}
|\dsl{g}{l}{n} \bbl{w}{l} | 
|\dsl{g}{l'}{n} \bbl{w}{l'} |^{\ppm-1} \,\dd x 
\to 0
\end{equation}
as $n\toinfty$.
We also have 
\begin{equation}\label{202105080080}
    \varlimsup_{L\to \Lambda} \varlimsup_{n\toinfty} 
    \| r^L_n \|_{L^{\ppm}(\RN)} =0.
\end{equation}

Combining~\eqref{202105080060}--\eqref{202105080080}, 
and passing to the limits as $n\toinfty$,  $L \to \Lambda$ and then $\eps \to 0$, 
one gets 
\begin{equation*}
\varlimsup_{L\to \Lambda} \varlimsup_{n\toinfty} 
\int_{\RN} \l|
F(u_n) - \sum_{l=0}^L F(\dsl{g}{l}{n} \bbl{w}{l} )
\r| \,\dd x 
=0.
\end{equation*}
By~\eqref{growth cond Wdmp 1} and the change of variables, we get 
\begin{align*}
&\varlimsup_{L\to \Lambda} \varlimsup_{n\toinfty} 
\l|  \int_{\RN} 
F(u_n)  \,\dd x 
-
\sum_{l=0}^L  \int_{\RN}   F(\bbl{w}{l} )   \,\dd x \r| \\ 
&\le 
\varlimsup_{L\to \Lambda} \varlimsup_{n\toinfty} 
\int_{\RN} \l|
F(u_n) - \sum_{l=0}^L F(\dsl{g}{l}{n} \bbl{w}{l}  )
\r| \,\dd x 
\\ 
&=0,
\end{align*}
which yields 
\begin{equation*}
\lim_{n \toinfty}
\int_{\RN} 
F(u_n) \,\dd x 
=
\sum_{l=0}^\Lambda  \int_{\RN}   F(\bbl{w}{l} ) \,\dd x.
\end{equation*}
This completes the proof. 
\end{proof}

\subsection{Proof of \cref{theorem;B-L 4}}
We now move on to the proof of \cref{theorem;B-L 4}.
In this case, too, we shall prove the result in the special case $|\alpha|=0$ and we abbreviate $f_\alpha$ to $f$ and so on.

\paragraph{Basic lemma}
Firstly, we provide important inequalities corresponding to \cref{lemma;10180009,lemma;10180030}. 

\begin{lemma}\label{lemma;10240020}
Let $f $ and $F$ be as in \cref{theorem;B-L 4}. 
Then for any $L \in \N$, 
there exists a constant $C=C_L>0$ such that  
for any $s_1,\ldots, s_L \in \R$,  
\begin{equation}\label{10240020}
    \l|F \l( \sum_{l=1}^L s_l \r) - \sum_{l=1}^L F(s_l) \r|
    \le 
    C \sum_{1\le l\neq k \le L} |s_l||s_k|^{\ppm-1}.
\end{equation}
Moreover, for any $\eps>0$,  there exists a constant $C=C_\eps>0$
such that for all $a,b \in \R$,  
\begin{equation}\label{10240030}
    |F(a+b) - F(a)  | \le \eps |a|^{\ppm} + C|b|^{\ppm}.
\end{equation}
Furthermore, for any $\eps>0$, there exists $C_\eps, C_L >0$ such that for all $s_l, r \in \R$ $(l=1,\ldots, L)$, 
\begin{align}
&\label{202105080100}
 \l| F \l( \sum_{l=1}^L s_l + r \r) -\sum_{l=1}^L F(s_l)  \r| 
\\
&\notag 
\le 
\eps \l| \sum_{l=1}^L s_l \r|^{\ppm} 
+ C_\eps |r|^{\ppm}
+ C_L \sum_{1\le l\neq l' \le L} |s_l||s_{l'}|^{\ppm-1}. 
\end{align}
\end{lemma}

\begin{proof}
The inequality~\eqref{10240020} will be proved in the same way as the proofs of~\cite[Lemmas~4.4--4.7]{Okumura2}.
The inequality~\eqref{10240030} follows from the first inequality with $L=2$ and the Young inequality. 
The inequality~\eqref{202105080100} is proved by the above two. 
\end{proof}

\paragraph{Main body}
Now we are in a position to prove \cref{theorem;B-L 4}. 

\begin{proof}[Proof of \cref{theorem;B-L 4}]
From a similar argument to the preceding proof of \cref{theorem;B-L Wdmp 1} together with \cref{lemma;10240020}, one obtains the following: 
\begin{equation}\label{10240120}
    \lim_{L\to \Lambda}\varlimsup_{n\toinfty}
    \int_{\RN} \l| F(u_n) 
    -  \sum_{l=0}^L F \l( \dsl{g}{l}{n} \bbl{w}{l}   \r)  \r| \,\dd x =0.
\end{equation}
When $l \in  \N_1$,  there exists a limit $\bbl{j}{l} \in \Z$ 
such that $\dsl{j}{l}{n} \to \bbl{j}{l}$ as $n\toinfty$.
So if $l\in\N_1$, then  by the dominated convergence theorem, one sees that 
\begin{align}
\label{10240130}
    &\int_{\RN}
    F(\dsl{g}{l}{n} \bbl{w}{l}  ) \,\dd x  \\
&\notag =
\int_{\RN}
F( 2^{\dsl{j}{l}{n}N/\ppm} 
\bbl{w}{l} (2^{\dsl{j}{l}{n}} (x))  ) \,\dd x  \\
&\notag =
\int_{\RN}
F( 2^{\bbl{j}{l}N/\ppm} 
\bbl{w}{l} (2^{ \bbl{j}{l} } (x))  ) \,\dd x  
+o(1)
\end{align}
as $n\toinfty$.
Similarly, the dominated convergence theorem implies that 
\begin{alignat}{2}
\label{10240140}
&\int_{\RN}
 F( 2^{\dsl{j}{l}{n}N/\ppm} 
\bbl{w}{l} (2^{\dsl{j}{l}{n}} (x - \dsl{y}{l}{n}))  ) \,\dd x  
\\
&\notag \qquad 
=
\int_{\RN}
F_\infty ( \bbl{w}{l} ) \,\dd x 
+o(1) &\quad &(l \in \N_2), \\ 
\label{10240150}
&\int_{\RN}
 F( 2^{\dsl{j}{l}{n}N/\ppm} 
\bbl{w}{l} (2^{\dsl{j}{l}{n}} (x - \dsl{y}{l}{n}))  )  \,\dd x 
\\
&\notag \qquad 
=
\int_{\RN}
F_0 ( \bbl{w}{l} ) \,\dd x 
+o(1) &\quad &(l \in \N_3),
\end{alignat}
as $n\toinfty$. 

Combining~\eqref{10240130},~\eqref{10240140} and~\eqref{10240150}, 
one obtains
\begin{align*}
    &\l| \int_{\RN} F(u_n)  \,\dd x 
    - \sum_{l\in\N_1, l \le L} \int_{\RN}
    F( 2^{\bbl{j}{l} N/\ppm} 
    \bbl{w}{l} (2^{j_{l,0}} (x))  ) \,\dd x  \r. \\
    &\qquad \l. - \sum_{l\in\N_2, l \le L} \int_{\RN}
    F_\infty ( \bbl{w}{l} )\,\dd x 
    - \sum_{l\in\N_3, l \le L} \int_{\RN}
    F_0 ( \bbl{w}{l} ) \,\dd x  \r| \\
    &\le 
    \l| \int_{\RN} F(u_n)  \,\dd x 
    -\sum_{ l=0}^L \int_{\RN}
    F(\dsl{g}{l}{n} \bbl{w}{l}  )  \,\dd x \r| +o(1)
\end{align*}
as $n\toinfty$. 
From the above and~\eqref{10240120}, one concludes that
\begin{align*}
 &\lim_{L\to \Lambda} \varlimsup_{n\toinfty}
\l| \int_{\RN} F(u_n) \,\dd x
- \sum_{l\in\N_1, l \le L} \int_{\RN}
F( 2^{\bbl{j}{l} N/\ppm} 
\bbl{w}{l} (2^{\bbl{j}{l}} (x))  ) \,\dd x \r. \\
&\l. \qq\qquad - \sum_{l\in\N_2, l \le L} \int_{\RN}
F_\infty ( \bbl{w}{l} ) \,\dd x
- \sum_{l\in\N_3, l \le L} \int_{\RN}
F_0 ( \bbl{w}{l} ) \,\dd x \r| \\
&=0,    
\end{align*}
which leads us to~\eqref{eq;B-L 4}.
This completes the proof. 
\end{proof}

\section*{Acknowledgment}
The author would like to thank his supervisor Goro Akagi for his great support and advice during the preparation of the paper. 
The author also wishes to express his gratitude to Michinori Ishiwata and Norihisa Ikoma for a lot of valuable comments on proofs and the organization of the paper. 
The author also wishes to thank editors and reviewers who gave many valuable comments to improve this paper.

\appendix

\section{Abstract theory of profile decomposition}\label{section:appendixA}
We here recall two fundamental theorems of profile decomposition in general reflexive Banach spaces provided in the author's previous paper~\cite{Okumura2}. The following theorem is employed to prove \cref{theorem;profile-decomp. in Wdmp}. 

\begin{theorem}[\noindent{\cite[Theorem 2.1]{Okumura2}}]\label{theorem;ProDeco in X}
Let $(X,\|\cdot\|_X)$ be a reflexive Banach space, let 
$(X, G)$ be a dislocation space with a dislocation group $G$, and let $(u_n)$ be a bounded sequence in $X$. 
Then there exist $\Lambda \in \N\cup\{0,+\infty\}$, a subsequence $(N(n)) \preceq (n)$, 
$\bbl{w}{l} \in X \ (l \in \NN^{<\Lambda+1})$
and 
$\dsl{g}{l}{N(n)} \in G \ (l \in \NN^{<\Lambda+1}, \ n \in \Z_{\ge l})$ 
such that 
\begin{equation*}
u_{N(n)} = \sum_{l=0}^L \dsl{g}{l}{N(n)}\bbl{w}{l}+r^L_{N(n)}, \quad L \in \NN^{<\Lambda +1}, \ n \in \Z_{ \ge L}, 
\end{equation*}
and the following hold:
\begin{enumerate}
\setlength{\itemsep}{0mm}
\setlength{\parskip}{0mm}

\item $\dsl{g}{0}{N(n)}={\rm Id}_X \ (n \ge 0), \quad \bbl{w}{l}\neq 0 \ (1 \le l \in \NN^{<\Lambda+1})$.
\item $\dsl{g}{l}{N(n)}^{-1}\dsl{g}{k}{N(n)} \to 0$ operator-weakly 
as $n \toinfty$ whenever $l\neq k \in \NN^{<\Lambda +1}$.
\item $\dsl{g}{l}{N(n)}^{-1}u_{N(n)} \to \bbl{w}{l}$ weakly in $X$ 
$(n\toinfty, \ l\in \NN^{<\Lambda +1})$.
\item
For $k\in\NN^{<\Lambda+1}$, 
\begin{equation*}
\dsl{g}{k}{N(n)}^{-1} r^{L}_{N(n)} \to
\begin{cases}
0 & \mbox{if}\ k=0,\ldots,L, \\
\bbl{w}{k} & \mbox{if}\ k\ge L+1,
\end{cases}
\end{equation*}
weakly in $X$ as $n\toinfty$.
\end{enumerate} 
Furthermore, if either $\Lambda =\infty$ and $\|\bbl{w}{l}\|_X \to 0$ as $l\to \Lambda$, or else $\Lambda<\infty$, then the following exactness condition holds:  
\begin{equation}\label{eq;Ishiwata condition X}
\lim_{L\to \Lambda}\sup_{\phi\in U}\varlimsup_{n\toinfty}\sup_{g\in G}
\qty| \la  \phi, g^{-1}r^L_{N(n)} \ra|=0,
\end{equation}
where $U \coloneqq B_{X^*}(1)$.
\end{theorem}

As is remarked in~\cite{Okumura2}, we put the same remark here:
\begin{remark}
\begin{itemize}

\item 
The above theorem gives \emph{qualitative} assertions of profile decomposition, and \emph{quantitative} one~\eqref{eq;Ishiwata condition X} is obtained under a further assumption that either $\Lambda =\infty$ and $\|\bbl{w}{l}\|_X \to 0$ as $l\to \Lambda$, or else $\Lambda<\infty$. 
However, this further assumption will be ensured in theorems of profile decomposition below in the present paper and in~\cite{Okumura2}, by virtue of the direct calculations for the decompositions in energy like~\eqref{eq;energy estim Wdmp}, which shows that $\norm{\bbl{w}{l}}_X\to 0$ as $l\toinfty$ if $\Lambda=\infty$. 

\item 
As for the assumption ``either $\Lambda =\infty$ and $\|\bbl{w}{l}\|_X \to 0$ as $l\to \Lambda$, or else $\Lambda<\infty$'', 
Solimini-Tintarev~\cite{SoliminiTintarev} generally verified this further assumption, by reformulating the profile decomposition theory by means of the so-called ``$\Delta$-convergence''. 
They established a $\Delta$-convergence-version of the profile decomposition theory for uniformly convex and uniformly smooth Banach spaces, where the above further assumption always holds true, and they also showed that if the Banach space satisfies Opial's condition, the $\Delta$-limits coincide with the weak limits, so that the above further assumption are also true with respect to the weak-topological profile decomposition theory. 
It is noteworthy that Hilbert spaces, Besov spaces and Triebel-Lizorkin spaces (including Sobolev spaces) enjoys Opial's condition, so that our assumption is satisfied in those cases.

\end{itemize}
\end{remark}


The following theorem implies that the residual term satisfying the exactness condition~\eqref{eq;Ishiwata condition X} becomes arbitrarily small in a normed space $Y$, where the embedding $X\hookrightarrow Y$ is $G$-completely continuous. From the following theorem, one can prove~\eqref{residue vanish Wdmp} and~\eqref{residue vanish Wdmp2}. 

\begin{theorem}[\noindent{\cite[Theorem 2.5]{Okumura2}}]\label{theorem;weak G-comp conti}
Let $(X,G)$ be as in \cref{theorem;ProDeco in X} and let $(Y,\|\cdot\|_Y)$ be a normed space.
Suppose that the embedding 
$X\hookrightarrow Y$ 
is $G$-completely continuous.
Also assume that a double-suffix sequence $(u^L_n)$ in $X$ satisfies that 
\begin{align*}
&\sup_{n,L\in\N}\|u^L_n\|_X<\infty, \\
&\lim_{L\to\infty}\sup_{\phi\in B_{X^*}(1)}
\varlimsup_{n\to\infty}\sup_{g\in G}
\qty| \la \phi,g^{-1}u^L_n\ra |=0.
\end{align*}
Then it holds that 
\begin{equation*}
\lim_{L\to\infty}\varlimsup_{n\to\infty}
\|u^L_n\|_Y=0.
\end{equation*}
\end{theorem}

\section{Abbreviated form of profile decomposition}\label{section:abbreviation}
In the profile decomposition theorem as before, we explicitly denoted the number of nontrivial profiles by $\Lambda\in\Z\cup\{0,\infty\}$, so that one can distinguish the infinite-profiles case from the finite-profiles case. 

In practice, however, many profile decomposition theorems use the abbreviated form as in the introduction. 
Here, we provide the abbreviated version of the profile decomposition theorem.

\begin{theorem}[Abbreviated form of~\cref{theorem;profile-decomp. in Wdmp}]\label{theorem:abbreviation}
Let $(u_n)$ be a bounded sequence in the dislocation Sobolev space $(\Wdmp, G[\RN,\R;\ppm])$. 
Then, there exist 
a subsequence of $(n)$, still denoted by $n$, 
profiles $\bbl{w}{l}\in \Wdmp$ $(l \in\NN)$, 
dislocations $(\dsl{y}{l}{n},\dsl{j}{l}{n} ) \in \RN\times\R$ $(l,n\in\NN)$, 
and residual terms $r^L_n\in \Wdmp$ $(L,n\in\NN)$,
with the relation of a double-suffix profile decomposition
\begin{equation}\notag 
u_{n} = \sum_{l=0}^L 
2^{\dsl{j}{l}{n}N/\ppm} \bbl{w}{l} \qty(2^{\dsl{j}{l}{n}} (\cdot - \dsl{y}{l}{n}))+r^L_{n} , 
\quad L,n \in \NN, 
\end{equation}
such that the following holds true:
\begin{enumerate}
\setlength{\itemsep}{0mm}
\setlength{\parskip}{0mm}

\item 
For all $n\in\NN$, 
$\dsl{y}{0}{n}=0$ and $\dsl{j}{0}{n}=0$.  
\item 
If $l\neq k$, then 
$|\dsl{j}{l}{n}-\dsl{j}{k}{n}|+
2^{\dsl{j}{k}{n}}
|\dsl{y}{l}{n}-\dsl{y}{k}{n}|
\toinfty$. 
\item 
For every $l\in\NN$, 
$2^{-\dsl{j}{l}{n}N/\ppm} u_{n} \qty( 2^{-\dsl{j}{l}{n}} \cdot + \dsl{y}{l}{n}) \to \bbl{w}{l}$ as $n\toinfty$ 
    weakly in  $\Wdmp$ and a.e.\  on $\RN$ $(l \in \NN)$. 
\item 
For every $k\in\NN$, 
\begin{equation*}
2^{-\dsl{j}{k}{n}N/\ppm} r^L_{n} \qty( 2^{-\dsl{j}{k}{n}} \cdot + \dsl{y}{k}{n})
\to 
\begin{cases}
0  &  \mbox{if} \ k=0,\ldots,L, \\
\bbl{w}{k}  &  \mbox{if} \ k\ge L+1, 
\end{cases}
\end{equation*}
weakly in $\Wdmp$ as $n\toinfty$. 
\item 
There holds 
\begin{align*}
\varlimsup_{n\toinfty}\|u_{n} \|_{\Wdmp}^p 
\ge \sum_{l=0}^\infty \|\bbl{w}{l}\|_{\Wdmp}^p 
+ \varlimsup_{L\to \infty} \varlimsup_{R\toinfty} \varlimsup_{n\toinfty}
\| r^L_{n} \|^p_{\dot{W}^{m,p}(\RN \setminus \Bnrl )  },   
\end{align*}
where $\Bnrl \coloneqq \bigcup_{l=0}^L 
B( \dsl{y}{l}{n},  2^{-\dsl{j}{l}{n}} R).$
\item 
There holds 
\begin{align*}
\lim_{L\toinfty}\sup_{\phi\in U}\varlimsup_{n\toinfty}\sup_{y\in\RN \!, \, j\in\R}
\abs{\la \phi, 2^{-jN/\ppm} r^L_n \qty(2^{-j}\cdot+y) \ra_{\Wdmp}} =0, 
\end{align*}
where $U \coloneqq B_{[\Wdmp]^*}(1)$.  
\end{enumerate}
\end{theorem}

\begin{proof}
Take a double-suffix profile decomposition given by~\cref{theorem;profile-decomp. in Wdmp}. 
At this time we renumber the subsequence and use the index $n$. 

If $\Lambda<+\infty$, then we set 
\begin{align*}
\bbl{w}{l}=0 \quad\text{for all}\s l\ge \Lambda+1, \\
\dsl{y}{l}{n}=0, \s \dsl{j}{l}{n}=0 \quad\text{for all}\s l \ge \Lambda+1, \ n\in\NN, \\ 
\dsl{y}{l}{n}=0, \s \dsl{j}{l}{n}=0  \quad\text{if}\s n<l.
\end{align*}
Then we get the above theorem. 

If $\Lambda=\infty$, then we set 
\begin{align*}
\dsl{y}{l}{n}=0, \s \dsl{j}{l}{n}=0  \quad\text{if}\s n<l.
\end{align*}
Then we get the above theorem. 
\end{proof}

\begin{corollary}
Let $u_n=\sum_{l=0}^L 2^{\dsl{j}{l}{n}N/\ppm} \bbl{w}{l} \qty(2^{\dsl{j}{l}{n}} (\cdot - \dsl{y}{l}{n}))+r^L_n$ $(n,L\in\NN)$ be a double-suffix profile decomposition as in the above theorem. 
Then this profile decomposition is the \emph{finite} profile decomposition if and only if 
there exists $L\in\NN$ such that 
$\bbl{w}{l}=0$ for all $l\ge L+1$. 
\end{corollary}

\section{Preliminaries}\label{section:preliminary}
We here provide some preliminary facts which are used in the paper. 

\paragraph{Dual spaces of Sobolev spaces}
We recall that duality structures of Sobolev spaces are well characterized via the Riesz representation theorem. 

\begin{lemma}[Dual space of $\Wdmp$]\label{lemma;dual Wdmp}
Let $1<p<N/m$ with  $m,N\in\N$.
Then for any 
$F\in [\Wdmp]^*$,  there exist 
$f_\alpha\in L^{p'}(\RN)$ $(\alpha\in(\NN)^N$ with $ |\alpha|=m)$ 
such that 
\begin{equation*}
\la F,u\ra =\sum_{|\alpha|=m} \int_{\RN} f_\alpha \d^\alpha u \,\dd x 
\mbox{\quad for all\s} u\in \Wdmp.
\end{equation*}
\end{lemma}

\begin{proof}
Let $M$ be the number of all multi-indices whose length is $m$, i.e., $M=\# \{\alpha \in (\NN)^N; \ |\alpha|=m \}$. Consider the product space $X \coloneqq L^{p}(\RN)^M$ equipped with the norm 
$\|u\|_X = ( \sum_{|\alpha|=m} \|u_\alpha\|_{L^{p}(\RN)}^{p} )^{1/p} $ for $u = (u_\alpha)_{|\alpha|=m} \in X$. 
Define an isometric injection $\kappa : \Wdmp \to X$ by  
$\kappa(u) = (\d^\alpha u)_{|\alpha|=m} \in X$, and set a closed subspace $Y \coloneqq \kappa (\Wdmp)$  of $X$. 

Let $F \in [\Wdmp]^*$. Then a bounded linear functional $\Phi \in Y^*$ is induced by $F$ as follows: for any $u \in Y$, $\la \Phi,u\ra  = \la F, \kappa^{-1} (u)\ra $, which is bounded and $\|\Phi\|_{Y^*} \le \|F\|_{(\Wdmp)^*}$. 
By the Hahn-Banach theorem, there is an extension $\hat{\Phi}$ of $\Phi$ onto whole $X$ such that $\|\hat{\Phi}\|_{X^*} = \|\Phi\|_{Y^*}$, i.e., $\hat{\Phi}\in X^*$. Then by the Riesz representation theorem, there exist $\phi_{\alpha} \in L^{p'}(\RN)$, $|\alpha|=m$, such that for all $u=(u_\alpha)_{|\alpha|=m} \in X$, 
\begin{equation*}
\la  \hat{\Phi}, u \ra  
=
\sum_{|\alpha|=m} \int_{\RN} \phi_\alpha u_\alpha \,\dd x. 
\end{equation*}
Therefore, we observe that for all $u \in \Wdmp$, 
\begin{equation*}
\la  F,u \ra 
=
\la  \Phi, \kappa(u) \ra 
=
\la  \hat{\Phi}, \kappa(u) \ra 
=
\sum_{|\alpha|=m} \int_{\RN} \phi_\alpha \d^\alpha u \,\dd x,
\end{equation*}
hence the conclusion.
\end{proof}

\paragraph{Besov spaces}
Here we recall homogeneous Besov spaces for the sake of the reader's convenience. 
For more details, we refer the reader to~\cite{G2}.
In what follows, we denote by $\mathcal{F}$ and $\mathcal{F}^{-1}$ 
the Fourier transformation and the  inverse Fourier transformation, respectively,
defined on $L^1(\RN), L^2(\RN)$ or $\mathcal{S}'(\RN)$,  
where $\mathcal{S}'(\RN)$ denotes the set of tempered distributions on $\RN$.

Let us recall the Littlewood-Paley decomposition of tempered distributions. 
Let $\{\phi_j\}_{j \in \Z}$ be a family in the Schwartz class satisfying the following properties:
\begin{align*}
&\supp \phi_j \subset \{\xi\in\RN; \ 2^{j-1} \le |\xi|\le 2^{j+1}\}; \\
&\begin{alignedat}{2}
&\sum_{j\in\Z} \phi_j (\xi) = 1, &\quad &\xi \neq 0; \\
&\phi_j (\xi)=\phi_0 (2^{-j}\xi), &\quad &\xi \in \RN, \s j\in \Z; \\
&\phi_{j-1}(\xi)+\phi_{j}(\xi)+\phi_{j+1}(\xi) = 1, &\quad &\xi \in \supp \phi_j, \s j\in\Z.
\end{alignedat}\\
\end{align*}
A family of operators $\{P_j\}$ is defined as follows:
\begin{equation*}
P_j u = \FF \phi_j \F u, \quad u\in \mathcal{S}'(\RN), \quad j \in \Z.
\end{equation*}
Then the Littlewool-Paley decomposition of $u\in \mathcal{S}'(\RN)$ is 
\begin{equation*}
    u = \sum_{j \in \Z} P_j u, \quad u \in \mathcal{S}'(\RN).
\end{equation*}

Under these settings, the homogeneous Besov space 
$\dot{B}^{s}_{p,q}(\RN)$ is defined as the set of 
all tempered distributions whose homogeneous Besov norm is finite; here 
the homogeneous Besov norm is defined as 
\begin{equation*}
\|u\|_{\dot{B}^{s}_{p,q}(\RN)}\coloneqq 
\|(\|2^{js} P_j u\|_{L^p(\RN)})_{j\in\Z}\|_{\ell^q(\Z)}
\end{equation*}
for $s \in \R, \ p\in [1,\infty], \ q\in [1,\infty]$ and $u\in \mathcal{S}'(\RN)$.


\end{document}